\documentclass[12pt]{article}

\usepackage[letterpaper, margin=1.7cm]{geometry}

\usepackage{amsmath, amsfonts, amssymb, amsthm}
\usepackage{setspace}
\usepackage{commath}
\usepackage{mathtools}

\usepackage{bbm}
\usepackage{bm}
\usepackage{subfig} 
\usepackage{color}
\usepackage[normalem]{ulem}
\usepackage{mathrsfs}
\usepackage{overpic}

\usepackage{cleveref}
\usepackage{enumerate}
\usepackage[singlelinecheck=false]{caption}

\def\ignore#1{}

\def\bR{{\mathbb R}}

\def\bN{\mathbb N}

\def\bT{\mathbb T}
\def\bS{\mathbb S}

\def\bP{{\mathbb P}}
\def\bE{{\mathbb E}}

\def\cF{{\mathscr F}}

\definecolor{DarkGreen}{rgb}{0.2,0.6,0.2}

\newcommand{\se}{ \subseteq}

\newcommand{\ee}{ \varepsilon }

\def\fp#1{\{#1\}}

\def\Ind#1{{\mathbbmss 1}_{_{\scriptstyle #1}}}
\def\lra{\longrightarrow}
\def\eps{\varepsilon}
\def\<{\langle}
\def\>{\rangle}
\def\wt#1{\widetilde{#1}}

 \def\lra{\longrightarrow}

 \def\ua{\uparrow}
 \def\da{\downarrow}

 \def\wt{\widetilde}
  
\newcommand{\kk}{ \kappa }

\newcommand{\bZ}{ {\mathbb{Z}} }
\newcommand{\n}[1]{\left\lVert#1\right\rVert}

\newcommand{\bone}{ {\mathbbm{1}} }

\def\ignore#1{}

\def\<{\langle}\def\>{\rangle}
\def\Ind#1{{\mathbbmss 1}_{_{\scriptstyle #1}}}

\newtheorem{theorem}{Theorem}[section]
\newtheorem{proposition}[theorem]{Proposition}
\newtheorem{corollary}[theorem]{Corollary}
\newtheorem{lemma}[theorem]{Lemma}

\theoremstyle{definition}
\newtheorem{definition}[theorem]{Definition} 
 
\newtheorem{remark}[theorem]{Remark}

\begin{document}
\title{\Large\bf WEIERSTRASS BRIDGES}
\author{\normalsize Alexander Schied\thanks{Department of Statistics and Actuarial Science, University of Waterloo. E-mail: {\tt aschied@uwaterloo.ca}}
	 \and\setcounter{footnote}{6}
\normalsize	Zhenyuan Zhang\thanks{
		Department of Mathematics, Stanford University. E-mail: {\tt zzy@stanford.edu}
	 		\hfill\break The authors gratefully acknowledge financial support  from the
 Natural Sciences and Engineering Research Council of Canada through grant RGPIN-2017-04054. We also thank Xiyue Han for valuable comments.}}
        
\maketitle

\vspace{-0.5cm}

\begin{abstract} We introduce a new class of stochastic processes called fractional Wiener--Weierstra\ss\ bridges. 
They arise by applying the convolution from the construction of the classical, fractal Weierstra\ss\ functions to an underlying fractional Brownian bridge.  By analyzing the $p$-th variation of the fractional Wiener--Weierstra\ss\ bridge along the sequence of $b$-adic partitions, we identify two regimes in which the processes exhibit distinct sample path properties. We also analyze the critical case between those two regimes for Wiener--Weierstra\ss\ bridges that are based on a standard Brownian bridge.  We furthermore prove that fractional Wiener--Weierstra\ss\ bridges are never semimartingales, and we show that their covariance functions are typically fractal functions. Some of our results are extended to Weierstra\ss\ bridges based on bridges derived from a general continuous Gaussian martingale. \end{abstract}
 
\noindent{\it Keywords:} Fractional Wiener--Weierstra\ss\ bridge, $p$-th variation, roughness exponent,  Gladyshev theorem, non-semimartingale process,  Gaussian process with fractal covariance structure \medskip

\noindent{\it MSC 2010:} 60G22, 60G15, 60G17, 28A80

\section{Introduction}


For $\alpha\in(0,1)$, $b\in\{2,3,\dots\}$, and a  continuous function $\phi: [0,1]\to\mathbb R$ with $\phi(0)=\phi(1)$, consider
\begin{equation}\label{f intro eq}
f(t):=\sum_{n=0}^\infty \alpha^n\phi(\fp{b^n t}),\qquad 0\le t\le 1,
\end{equation}
where $\fp{x}$ denotes the fractional part of $x\ge0$. 
If $\phi$ is a convex combination of trigonometric functions such as $\sin(2\pi t)$ or $\cos(2\pi t)$, we get Weierstra\ss' celebrated example~\cite{Weierstrass} of a function that is continuous but nowhere differentiable provided that $\alpha  b$ is sufficiently large. If $\phi$ is the tent map, i.e., $\phi(t)=t\wedge(1-t)$, then we obtain the class of Takagi--van der Waerden functions~\cite{Takagi,vanderWaerden}. Also, the case of a general Lipschitz continuous function  $\phi$ has been studied extensively; see, e.g., the survey~\cite{BaranskiSurvey} and the references therein. Typical questions that have been investigated include smoothness versus nondifferentiability~\cite{Hardy1916}, local and global moduli of continuity~\cite{deLimaSmania}, Hausdorff dimension of the graphs~\cite{Ledrappier,ren2021dichotomy}, extrema~\cite{Kahane}, and $p$-th and $\Phi$-variation~\cite{SchiedZZhang,HanSchiedZhang1}, to mention only a few. An intriguing connection between Weierstra\ss'   function and fractional Brownian motion is discussed in \cite{pipiras2000convergence}, where it is shown that a randomized version of  Weierstra\ss'   function converges to fractional Brownian motion.

In this paper, our goal is to study random functions that arise when $\phi$ is replaced by the sample paths of a stochastic process $B=(B(t))_{0\le t\le1}$  with identical values at $t=0$ and $t=1$. This leads to a new class of stochastic processes $X$ of the form
$$X(t):=
\sum_{n=0}^\infty \alpha^nB(\fp{b^n t}),\qquad 0\le t\le 1,
$$
that we call \emph{Weierstra\ss\ bridges}. 
In this paper, we mainly focus on Weierstra\ss\ bridges that are based on (fractional) Brownian bridges $B$. They are called \emph{(fractional) Wiener--Weierstra\ss\ bridges}.

Our first results study the $p$-th variation of the fractional Wiener--Weierstra\ss\ bridge $X$ along the sequence of $b$-adic partitions. Letting $H$ denote the Hurst parameter of $B$ and $K:=1\wedge(-\log_b\alpha)$, we show that the $p$-th variation of $X$ is infinite for $p<1/(H\wedge K)$ and zero for $p>1/(H\wedge K)$. The behavior of the $p$-th variation for $p=1/(H\wedge K)$ depends on whether $H<K$, $H=K$, or $H>K$. \Cref{main thm} identifies  this  $p$-th variation for  $H\neq K$. The critical case $H=K$ is more subtle and analyzed in \Cref{critical case H=1/2 thm}  for the case $H=1/2=K$. It contains a Gladyshev-type theorem for the rescaled quadratic variations of $X$, which implies that the quadratic variation itself is infinite.

We also show that the (fractional) Wiener--Weierstra\ss\ bridge is never a semimartingale and that its covariance function often has a fractal structure, which sometimes is just as \lq rough\rq\ as the sample paths of the process itself. We also briefly discuss the case in which the underlying bridge $B$ is derived from a generic, continuous Gaussian martingale. All our main results are presented in \Cref{resultssection}. The proofs are collected in \Cref{Proofs section}.

Some of our proofs are based on an analysis of deterministic fractal functions $f$ of the form \eqref{f intro eq}, for which $\phi$ is no longer Lipschitz-continuous but has H\"older regularity. These results are presented in \Cref {Hoelder section} and are of possible independent interest.

\section{Statement of main results}\label{resultssection}

Let $W=(W(t))_{t\ge0}$ be a fractional Brownian motion with Hurst parameter $H\in(0,1)$, and choose a deterministic function $\kappa:[0,1]\to[0,1]$ that satisfies $\kappa(0)=0$ and $\kappa(1)=1$. The stochastic process 
\begin{align}
B(t):=W(t)-\kappa(t)W(1),\qquad  t\in[0,1],\label{fbb}
\end{align}
can then be regarded as a fractional Brownian bridge. If we take specifically
\begin{align}\kappa(t):=\frac{1}{2}(1+t^{2H}-(1-t)^{2H}),\label{standard kappa}
\end{align}
then the law of $B$ is (at least informally) equal to the law of $W$ conditioned on $\{W(1)=0\}$, and so $B$ is a standard bridge; see~\cite{GasbarraSottinenValkeila}. However, the specific form of $\kappa$ is not going to be needed in the sequel. All we are going to require is that $B$ is of the form $B(t)=W(t)-\kappa(t)W(1)$ for some function $\kappa:[0,1]\to[0,1]$ that satisfies $\kappa(0)=0$ and $\kappa(1)=1$ and that is H\"older continuous with some exponent $\tau\in(H,1]$. Obviously, the function $\kappa$ in \eqref{standard kappa} satisfies these requirements. Both $B$ and $\kappa$ will be fixed in the sequel.

\begin{definition}  We denote by $\fp{x}$  the fractional part of $x\ge0$. For $\alpha\in(0,1)$ and $b\in\{2,3,\dots\}$, the stochastic process
\begin{equation}\label{WW def eq}
X(t):=
\sum_{n=0}^\infty \alpha^nB(\fp{b^n t}),\qquad 0\le t\le 1,
\end{equation}
is called the \emph{fractional Wiener--Weierstra\ss\ bridge} with parameters $\alpha$, $b$, and $H$.
\end{definition}

Our terminology stems from the fact that by replacing in \eqref{WW def eq} the fractional Brownian bridge with a 1-periodic trigonometric function, $X$ becomes a classical Weierstra\ss\ function. 
In addition, the following remark gives a representation of $X$ in terms of rescaled Weierstra\ss\ functions if $H>1/2$.

\begin{remark}\label{Fourier remark} Let us develop a sample path of $B$ into a Fourier series, i.e.,
\begin{equation}\label{BB Fourier series}
\begin{split}
B(t)
&=\sum_{k=1}^\infty\Big(\xi_k\big(\cos(2\pi kt)-1\big)+\eta_{k}\sin(2\pi kt)\Big),
\end{split}
\end{equation}
where we have used that $B(1)=0$ and where the $\xi_k$ and $\eta_k$ are certain centered normal random variables. If $H>1/2$, then $B$ is $\bP$-a.s.~H\"older continuous for some exponent larger than $1/2$, and so a theorem by Bernstein (see, e.g., Section I.6.3 in~\cite{Katznelson}) yields that the Fourier series~\eqref{BB Fourier series}
 converges absolutely, and in turn uniformly (see, e.g., Corollary 2.3 in~\cite{SteinShakarchi}). We therefore may interchange summation in 
\eqref{WW def eq} and obtain for $H>1/2$ the representation
$$X(t)=\sum_{k=1}^\infty\big(\xi_kf( kt)+\eta_{k}g(kt)\big),
$$
where
\begin{equation}\label{Weierstrass expansion functions}
f(t)=\sum_{n=0}^\infty\alpha^n\big(\cos(2\pi b^nt)-1\big)\qquad\text{and}\qquad g(t)=\sum_{n=0}^\infty\alpha^n\sin(2\pi b^nt)
\end{equation}
are classical Weierstra\ss\ functions. 

 \end{remark}

The sample paths of $X$ have two competing sources of \lq roughness\rq. The first is due to the underlying fractional Brownian bridge, whose roughness is usually measured by the Hurst parameter $H$. The second source is the Weierstra\ss--type convolution, which generates fractal functions. In the context of \Cref{Fourier remark}, the latter source can also be represented through the roughness of the Weierstra{\ss} functions $f$ and $g$ in \eqref{Weierstrass expansion functions}. For fractional Brownian motion, the Hurst parameter, which is originally defined via autocorrelation, also governs many sample path properties \cite{Mishura} and is thus an appropriate measure of the roughness of trajectories. 
However, as pointed out in \cite{Gneiting2004Hurst}, the Hurst parameter of a given stochastic process may sometimes be completely unrelated to a geometric measure of roughness such as the fractal dimension. As discussed in more detail in~\cite{HanSchiedHurst}, a more robust approach to measuring the roughness of a function $f:[0,1]\to\bR$ is based on the concept of the $p$-th variation of $f$ along a refining sequence of partitions. Based on the sequence of  $b$-adic partitions, which will be fixed throughout this paper,  the $p$-th variation of $f$ is defined as 
\begin{equation}\label{pth variation}
\<f\>^{(p)}_t:=\lim_{n\ua\infty}\sum_{k=0}^{\lfloor tb^n\rfloor}\big|f((k+1)b^{-n})-f(kb^{-n})\big|^p,\qquad t\in[0,1],
\end{equation}
provided the limit exists for all $t$ and where $\lfloor x\rfloor$ denotes the largest integer less than or equal to $x$. This concept of $p$-th variation is meaningful for several reasons. First, functions that admit a continuous  $p$-th variation can be used as integrators in pathwise  It\^o calculus even if they do not arise as typical trajectories of a semimartingale; this fact was first discovered by F\"ollmer~\cite{FoellmerIto} for $1\le p\le 2$ and more recently extended to all $p\ge1$ by Cont and Perkowski~\cite{ContPerkowski}. Second, the following implication holds for $t>0$,
\begin{equation}\label{qth variation when pth variation nontrivial eq}
\text{ if $0<\<f\>^{(p)}_t<\infty$, then \ }\<f\>^{(q)}_t=\begin{cases}
\infty&\text{for $q<p$,}\\
0&\text{for $q>p$;}
\end{cases}
\end{equation}
see the final step in the proof of Theorem 2.1 in~\cite{MishuraSchied2}.
Thus, if $p$ is such that \eqref{qth variation when pth variation nontrivial eq}
 holds, then $K:=1/p$ is a natural measure for the roughness of $f$. It is called the roughness exponent in~\cite{HanSchiedHurst}. For our fractional Brownian bridge, we have $\bP$-a.s.~that $\<B\>^{(1/H)}_t=t\cdot\bE[|W(1)|^{1/H}]$ (this follows from combining~\cite[Theorem 5.1]{HanSchiedHurst} with~\cite[Lemma 2.4]{SchiedZZhang}), and so its roughness exponent is almost surely equal to the  Hurst parameter $H$. For Weierstra\ss\ functions of the form \eqref{Weierstrass expansion functions}, it is a consequence of Theorem 2.1 in~\cite{SchiedZZhang} that their roughness exponent is given by
\begin{equation*}
K=1\wedge\big({-\log_b\alpha}\big).
\end{equation*}

When analyzing the roughness of the trajectories of the fractional Wiener--Weierstra\ss\ bridge, we can expect competition between the Hurst exponent $H$ of the underlying fractional Brownian bridge and the roughness exponent $K$ resulting from the Weierstra\ss--type convolution. Indeed, our first result,
\Cref{main thm}, confirms in particular that the roughness exponent of the sample paths of $X$ is given by $H\wedge K$, provided that $H\neq K$. It shows moreover that for $p=1/(H\wedge K)$, the  $p$-th variation of $X$ has distinct features in each of the two regimes $H<K$ (Hurst exponent wins the competition) and $H>K$ (Weierstra\ss--type convolution wins the competition). For $H<K$, the trajectories of $X$ have deterministic $p$-th variation that we can compute explicitly.  For $H>K$, however, the  $p$-th variation of $X$ appears to be no longer deterministic. The critical case $H=K$ is more delicate and will be discussed subsequently.
 
\begin{theorem}\label{main thm} Let $X$ be a fractional Wiener--Weierstra\ss\ bridge with parameters $\alpha$, $b$, and $H$, and suppose that $H\neq K=1\wedge(-\log_b\alpha)$.
Then $\bP$-almost every sample path of $X$ admits the roughness exponent $H\wedge K$. More precisely:
\begin{enumerate}
 \item For $H>K$, there exists a finite and strictly positive random variable $V$ such that $\bP$-a.s.~for all $t\in[0,1]$,
 \begin{equation}\label{thm H>K}
 \<X\>^{(1/K)}_t=V\cdot t.
 \end{equation}
%
 \item For $H<K$, we have $\bP$-a.s.~for all $t\in[0,1]$,
 \begin{equation}\label{qua}
\<X\>^{(1/H)}_t= \bigg(\frac{2^{1/(2H)}\Gamma\big(\frac{H+1}{2H}\big)}{\sqrt\pi(1-\alpha^2b^{2H})^{1/(2H)}}\bigg)\cdot t.
 \end{equation}
\end{enumerate}
\end{theorem}

The factor $\frac1{\sqrt\pi}2^{1/(2H)}\Gamma(\frac{H+1}{2H})$ appearing in \eqref{qua}  is equal to $\bE[|Z|^{1/H}]$, where $Z$ is a standard normal random variable. It can be viewed as the contribution of $B$ to $\<X\>^{(1/H)}$. The term $(1-\alpha^2b^{2H})^{1/(2H)}$, on the other hand, results from the Weierstra\ss--type convolution in the construction of $X$. The random variable $V$ in \eqref{thm H>K} has a complicated structure. As we are going to see in \Cref{pth var rep remark}, $V$ can be represented as a mixture of the $(1/K)^{\text{th}}$ powers of the absolute values of certain Wiener integrals with integrator $W$. The histograms in \Cref{V figure} provide an illustration of the empirical distribution of $V$ for two sets of parameter values.
\begin{figure}
\centering
\includegraphics[width=8cm]{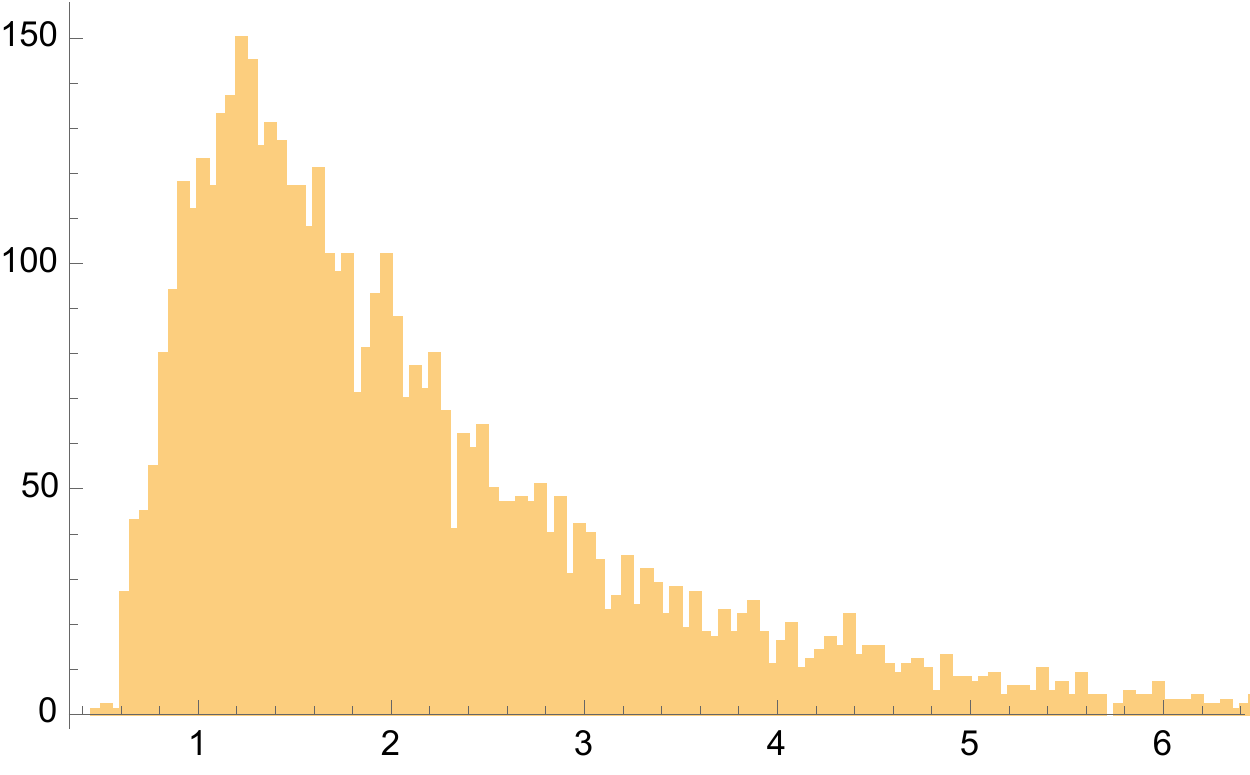}\qquad \includegraphics[width=8cm]{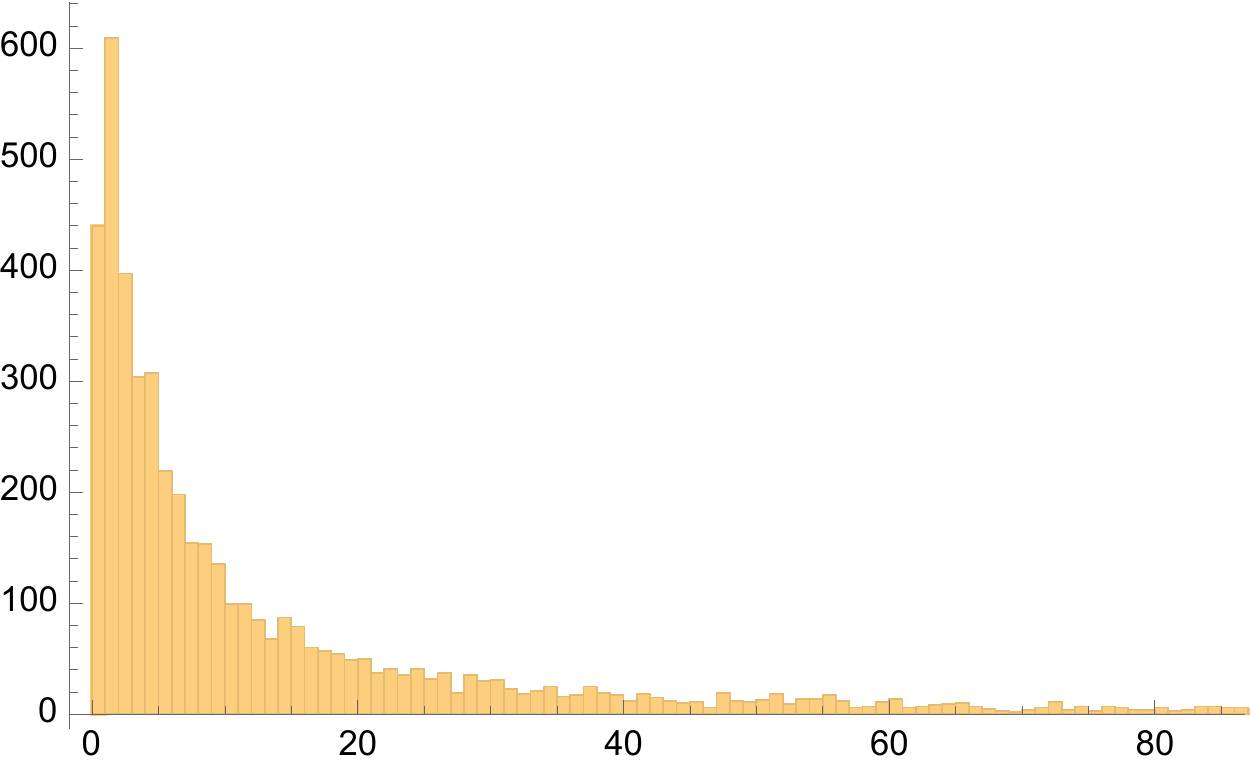}
\caption{Illustration of \Cref{main thm} (a) by means of histograms of the $(1/K)^{\text{th}}$ variation $\sum_{k=0}^{b^n}|X((k+1)b^{-n})-X(kb^{-n})\big|^{1/K}$ for 5000 sample paths of the fractional Wiener--Weierstra\ss\ bridge with $n=16$, $b=2$, and parameters $H=0.7$ and $K=0.5$ (left) versus $H=0.5$ and $K=0.2$ (right).}
\label{V figure}\end{figure}

Note that the increment process of the fractional Wiener--Weierstra\ss\ bridge is highly nonstationary. Therefore, classical results on the variation of Gaussian processes with stationary increments \cite{marcus1992p,MarcusRosen} are not applicable. 

Now we turn to the critical case  $H=K$, which we discuss for  $H=1/2$. In this case, the pattern observed in \Cref{main thm} breaks down and the quadratic variation of $X$ is infinite, even though the roughness exponent of $X$ is still equal to $1/2$. 

\begin{theorem}\label{critical case H=1/2 thm} Let $X$ be a  Wiener--Weierstra\ss\ bridge with parameters $\alpha$, $b$, and $H=1/2$ where $\alpha^2b=1$ $($in particular, $H= 1/2=K=1\wedge(-\log_b\alpha))$. Then the roughness exponent of $X$ is almost surely equal to $1/2$. Moreover,  $\bP$-a.s.~for all $t\in(0,1]$,
\begin{equation}\label{critical case H=1/2 eqn} 
\lim_{n\ua\infty}\frac1n \sum_{k=0}^{\lfloor tb^n\rfloor}\big(X((k+1)b^{-n})-X(kb^{-n})\big)^2=  t.
\end{equation}
In particular, the $p$-th variation of $X$ is almost surely infinite for $p\le2$ and zero for $p>2$.
\end{theorem}

Since the quadratic variation of the Wiener--Weierstra\ss\ bridge with $H=1/2=K$ is infinite, it cannot be a semimartingale. The following theorem extends the latter observation to all parameter choices.

\begin{theorem}\label{semimartingale thm}
For any $H\in(0,1)$, $\alpha\in(0,1)$, and $b\in\{2,3,\dots\}$, the fractional Wiener--Weierstra\ss\ bridge $X$ is not a semimartingale.
\end{theorem}

Let us discuss another aspect of \Cref{critical case H=1/2 thm}.
The convergence of the rescaled quadratic variations in 
\eqref{critical case H=1/2 eqn} can be regarded as a Gladyshev-type theorem for the Wiener--Weierstra\ss\ bridge. As in the original work by Gladyshev \cite{Gladyshev}, for the derivation of such results it is commonly assumed that the covariance function 
\begin{equation*}
c(s,t):=\text{\rm cov}(X(s),X(t))
\end{equation*}
of the  stochastic process $X$ satisfies certain differentiability conditions; see also \cite{KleinGine}.  However, the following result states that the covariance function of the fractional Wiener--Weierstra\ss\ bridge is often itself a fractal function; see Figure \ref{cov fig} for an illustration. For the particular case $H=1/2=K$ investigated in \Cref{critical case H=1/2 thm}, the following result implies in particular that $t\mapsto c(1/2,t)$ is a nowhere differentiable Takagi--van der Waerden function. Thus, \Cref{critical case H=1/2 thm} might also be interesting as a case study for Gladyshev-type theorems without smoothness assumptions. 

\begin{proposition}\label{covthm}
Suppose that $B$ is the standard fractional Brownian bridge with $\kappa$ given by \eqref{standard kappa} and that $K={-\log_b\alpha}< (2H)\wedge 1$. Then, for all $s\in(0,1)$,  the covariance function $c(s,t)=\text{\rm cov}(X(s),X(t))$ is such that $t\mapsto c(s,t)$ has finite, nonzero, and linear $(1/K)^{\text{th}}$ variation.  Moreover, if $b$ is even,  $H=1/2$, and $s=1/2$, then the function $t\mapsto 2c(1/2,t)$ is the Takagi--van der Waerden function  with parameters $b$ and  $\alpha$, that is, the function in \eqref{f intro eq} for the tent map $\phi(t)=t\wedge(1-t)$.
\end{proposition}

\begin{figure}[h]
\begin{center}
\includegraphics[width=7.5cm]{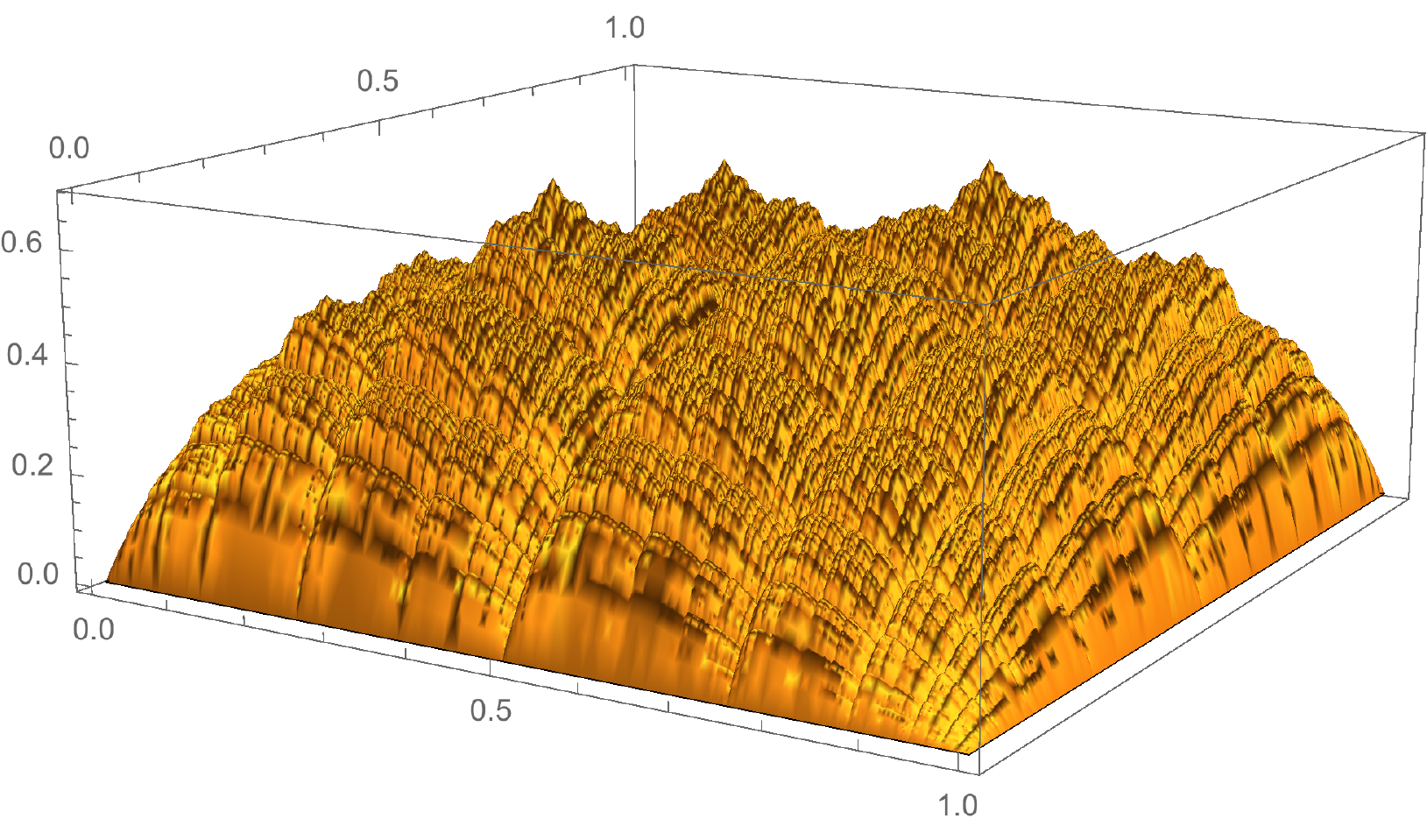}\qquad\qquad \includegraphics[width=7.5cm]{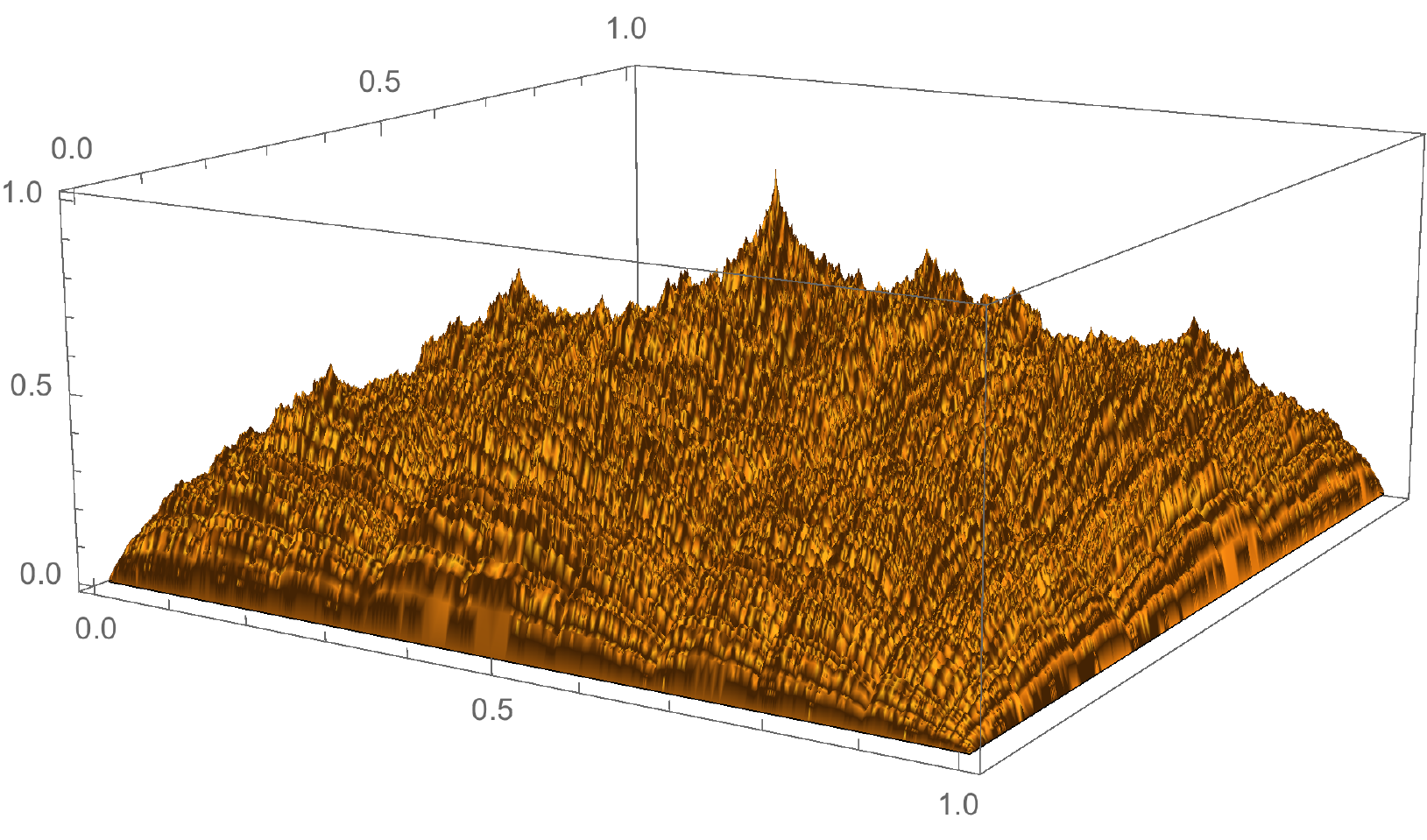}
\end{center}
\caption{Covariance functions of the Wiener--Weierstra\ss\ bridge for $H=1/2$, $\alpha=1/2$, and $b=2$ (left) and $b=3$ (right).}
\label{cov fig} \end{figure}

A remarkable consequence of  Theorem \ref{main thm} and Proposition \ref{covthm}  is that for every $K\in(0,1)$ there is a Gaussian process whose sample paths and covariance function both admit the roughness exponent $K$.  On the other hand, there exists no centered Gaussian process whose sample paths have a strictly lower roughness exponent than its covariance function. A precise statement of these facts is given in the following corollary.

 \begin{corollary}\label{co}
Suppose that $p>1$ and $b\in\{2,3,\dots\}$. \begin{enumerate}
\item There exists a centered Gaussian process $Y$ indexed by $[0,1]$ whose sample paths have $\bP$-a.s.~finite, nontrivial, and linear $p$-th variation and, for any $s\in(0,1)$, the covariance $t\mapsto c(s,t):=\bE[Y(s)Y(t)]$ has finite, nontrivial, and linear  $p$-th variation.
\item Suppose that $Y$ is any centered Gaussian process  indexed by $[0,1]$ whose sample paths have $\bP$-a.s.~finite (though not necessarily nonzero) $p$-th variation.  Then, for every $s\in[0,1]$, 
\begin{equation}\label{cov pth var eq}
\limsup_{n\ua\infty}\sum_{k=0}^{b^n-1}\Big|c\big(s,(k+1)b^{-n}\big)-c\big(s,kb^{-n}\big)\Big|^p<\infty.\end{equation}
\end{enumerate}
\end{corollary}

\begin{remark}We will see in \Cref{Hoelder prop} that the sample paths of the fractional Wiener--Weierstra\ss\ bridge are $\bP$-a.s.~H\"older continuous with exponent $K$ for $K<H$. If $K\ge H$, then the trajectories are $\bP$-a.s.~H\"older continuous with exponent $\gamma$ for every 
$\gamma<H$. 
\end{remark}

Weierstra\ss\ bridges are not limited to fractional Brownian bridges. They can be defined and studied for a large class of underlying bridges. In the sequel, we are going to illustrate this by the Gaussian bridges arising from a continuous Gaussian martingale $M=(M(t))_{t\in[0,1]}$ with $M(0)=0$ and $\text{var}(M(1))=\bE[(M(1))^2]>0$. When letting 
$$\kappa(t):=\frac{\text{cov}(M(t),M(1))}{\text{var}(M(1))},
$$
 the law of the Gaussian bridge
$$B(t)=M(t)-\kappa(t)M(1),\qquad 0\le t\le1,
$$
is (at least informally) equal to the distribution of $M$ conditional on $\{M(1)=0\}$ (by~\cite{GasbarraSottinenValkeila}, this holds for any continuous Gaussian process; \eqref{standard kappa} is an example). Since Gaussian martingales have deterministic quadratic variation,  $\kappa$ can also be represented as follows,
$$\kappa(t)=\frac{\<M\>_t}{\<M\>_1}.
$$
Now we take $\alpha\in(0,1)$ and $b\in\{2,3,\dots\}$ and let 
\begin{equation}\label{Gaussian Weierstrass bridge eq}
X(t):=
\sum_{n=0}^\infty \alpha^nB(\fp{b^n t}),\qquad 0\le t\le 1.
\end{equation}
The process $X$ will be called the \emph{Gaussian Weierstra\ss\ bridge} associated with $B$ and with parameters $\alpha$ and $b$.

\begin{proposition}\label{Gaussian martingale bridge prop}Suppose that  $\<M\>_t=\int_0^t\varphi(s)\,ds$ for a bounded measurable function $\varphi:[0,1]\to[0,\infty)$ for which there exists a nonempty open interval $I\subset[0,1]$ such that $\varphi>0$ on $I$. Suppose moreover that $K={-\log_b\alpha}<1/2$.  Then the corresponding Gaussian Weierstra\ss\ bridge \eqref{Gaussian Weierstrass bridge eq} has finite, nonzero, and linear $(1/K)$-th variation along the $b$-adic partitions.
\end{proposition}

\paragraph{Outlook and some open questions.} Let us conclude this section by pointing out some interesting open questions and directions for future research that are beyond the scope of this paper. 
\begin{enumerate}[{\rm 1.}]
\item \Cref{V figure} suggests that the random variable $V$ in part (a) of \Cref{main thm}  is a nondegenerate random variable with nonzero variance. Establishing this claim would provide a distinctive feature between the two regimes $H<K$ and $H>K$.
\item \Cref{critical case H=1/2 thm} analyzes only the case $H=1/2$, and it would be interesting to obtain a similar result for all  $H\in(0,1)$. Based on the argument used to establish the convergence of expectation in \Cref{proof of critical case H=1/2 thm section}, we expect that  the following convergence holds for arbitrary $H\in(0,1)$ and  $\bP$-a.s.~for all $t\in[0,1]$,
\begin{equation*}
\lim_{n\ua\infty}\frac1{n^{1/(2H)}} \sum_{k=0}^{\lfloor tb^n\rfloor}\big|X((k+1)b^{-n})-X(kb^{-n})\big|^{1/H}= \frac{2^{1/(2H)}\Gamma(\frac{H+1}{2H})}{\sqrt{\pi}}\cdot t.
\end{equation*}
\item \Cref{main thm}  shows that the sample paths of $X$ have finite, nonzero, and linear $(H\wedge K)^{-1}$-th variation if $H\neq K$. \Cref{critical case H=1/2 thm}, on the other hand, shows that this relation breaks down in the critical case $H=K$. A similar phenomenon appears already in the deterministic case \eqref{f intro eq} with Lipschitz continuous $\phi$. In this case, $H=1$ and we pick  $\alpha=1/2$ and $b=2$ in \eqref{f intro eq}, so that also $K=1$. Yet, the corresponding function \eqref{f intro eq} is typically not of bounded variation and, in the case where $\phi$ is the tent map, has been used as a classical example of a nowhere differentiable function \cite{Takagi}. It was shown in \cite{HanSchiedZhang1} that Weierstra\ss\ and Takagi--van der Waerden functions with critical roughness possess finite, nonzero, and linear $\Phi$-variation  for the function $\Phi(x)=x(-\log x)^{-1/2}$, where $\Phi$-variation is understood in the Wiener--Young sense and taken along the $b$-adic partitions. We expect that a similar effect is in play for fractional Wiener--Weierstra\ss\ bridges and conjecture that for $H=K$ and  $\Phi(x)=x^{1/H}(-\log x)^{-1/(2H)}$,  $\bP$-a.s.~for all $t\in[0,1]$,
 \begin{equation*}
 \lim_{n\ua\infty}\sum_{k=0}^{\lfloor tb^n\rfloor}\Phi\big(\big|X((k+1)b^{-n})-X(kb^{-n})\big|\big)=\frac{2^{1/(2H)}\Gamma\big(\frac{H+1}{2H}\big)}{\sqrt\pi(-\log\alpha)^{1/(2H)}}\cdot t.
 \end{equation*}
 \item As pointed out by an anonymous referee, an interesting problem would be studying the fluctuations of the limits \eqref{qua} and \eqref{critical case H=1/2 eqn}, as a function of $n$, the number of steps taken in the $b$-adic partition. For example, we may ask whether central limit theorems and large deviation principles can be established.

 

\end{enumerate}

\section{Proofs}\label{Proofs section}

Let $X$ denote the fractional Wiener--Weierstra\ss\ bridge with parameters $\alpha$, $b$, and $H$ and recall that the underlying fractional Brownian bridge is of the form $B(t)=W(t)-\kappa(t)W(1)$ for a fractional Brownian motion $W$ with Hurst parameter $H$ and a function $\kappa:[0,1]\to [0,1]$ that is H\"older continuous with exponent $\tau\in(H,1]$ and satisfies $\kappa(0)=0$ and $\kappa(1)=1$. Moreover, the parameters $K$ and $p$ are defined as
$$K=1\wedge(-\log_b\alpha)\quad\text{and}\quad p=\frac1{(K\wedge H)}.
$$

Let us first discuss the  $p$-th variation of a general function $f$ of the form  \eqref{f intro eq}. In \eqref{pth variation},  this  $p$-th variation  is defined as the limit of the terms
\begin{align}
V_n&:=\sum_{k=0}^{b^n-1}\big|f((k+1)b^{-n})-f(kb^{-n})\big|^p=\sum_{k=0}^{b^n-1}\bigg|\sum_{m=0}^{n-1} \alpha^{m}\Big(\phi(\fp{(k+1)b^{m-n}})-\phi(\fp{kb^{m-n}})\Big)\bigg|^p\nonumber\\
&=\alpha^{np}\sum_{k=0}^{b^n-1}\bigg|\sum_{m=1}^{n} \alpha^{-m}\Big(\phi(\fp{(k+1)b^{-m}})-\phi(\fp{kb^{-m}})\Big)\bigg|^p,\label{Vn rep}
\end{align}
where we have used the assumption $\phi(0)=\phi(1)$ in the second step.
Following \cite{SchiedZZhang}, 
we now let $(\Omega_R,\mathscr{F}_R,\mathbb{P}_R)$ be a probability space supporting an independent sequence $U_1,U_2,\dots$  of random variables with 
a uniform distribution on $\{0,1,\dots, b-1\}$ and define the stochastic process
\begin{equation}\label{R and X}
R_m:=\sum_{i=1}^mU_ib^{i-1},\qquad m\in\mathbb{N}.
\end{equation}
The importance of the random variables $R_m$ and the need to have the random variables $R_m$ defined independently of our underlying Gaussian processes explain the subscript \lq$R$' in our notation $(\Omega_R,\mathscr{F}_R,\mathbb{P}_R)$.
Note that each $R_m$ has a uniform distribution on $\{0,\dots, b^m-1\}$. Moreover, \eqref{R and X}
 ensures that
\begin{equation}\label{RnRm}
\fp{b^{-m}R_n}=b^{-m}R_m\qquad\text{ for $m\le n$.}
\end{equation} 
Following \cite{SchiedZZhang}, we can now express the sum over $k$  in \eqref{Vn rep} through an expectation over $(R_m)_{1\leq m\leq n}$ and then use \eqref{RnRm} to obtain
\begin{align}
V_n&=(\alpha^pb)^n\bE_R\bigg[\bigg|\sum_{m=1}^n \alpha^{-m}(\phi((R_m+1)b^{-m})-\phi(R_mb^{-m}))\bigg|^p\bigg].\label{Vn Rm eq}
\end{align}
 
 For notational clarity, the probability space on which the fractional Brownian motion $W$ and its corresponding bridge $B$ are defined will henceforth be denoted by  $(\Omega_W,\mathscr{F}_W,\mathbb{P}_W)$. The product space of these two probability spaces will be denoted by $(\Omega,\mathscr{F},\mathbb{P})$. That is, $\Omega=\Omega_W\times \Omega_R$, $\cF=\cF_W\otimes\cF_R$, and $\bP=\bP_W\otimes\bP_R$.

\subsection{Proof of \Cref{main thm} (a)}

In this section, we will deploy \Cref{p variation prop} in \Cref{Hoelder section} to prove part (a) of \Cref{main thm}. Recall that this part addresses the situation in which $K<H$, which is equivalent to either $\alpha b^H>1$ or $\alpha^pb>1$. 

\begin{lemma}\label{fractional bridge thm lemma}
In the context of \Cref{main thm} {\rm (a)}, let $(r_m)_{m\in\bN}$ be an arbitrary sequence of integers such that $r_m\in\{0,\dots,b^m-1\}$. Then the expectation
\begin{equation}\label{sum B rm eq}
\bE_W\bigg[\bigg(\sum_{m=1}^\infty\alpha^{-m}\Big(B\Big(\frac{r_m+1}{b^m}\Big)-B\Big(\frac{r_m}{b^m}\Big)\Big)\bigg)^2\bigg]
\end{equation}
is finite and strictly positive.
\end{lemma}

\begin{proof}For $m,n\in\mathbb N$, let us define  $\Delta_m\kappa:=\kappa(\frac{r_m+1}{b^m})-\kappa(\frac{r_m}{b^m})$ and
$$
f_n:=\sum_{m=1}^n\alpha^{-m}\Ind{[\frac{r_m}{b^m},\frac{r_m+1}{b^m}]}(t) -\sum_{m=1}^n \alpha^{-m}\Delta_m\kappa.
$$
Using $B(t)=W(t)-\kappa(t)W(1)$, we get
\begin{align}
\sum_{m=1}^n\alpha^{-m}\Big(B\Big(\frac{r_m+1}{b^m}\Big)-B\Big(\frac{r_m}{b^m}\Big)\Big)
&=\sum_{m=1}^n\alpha^{-m}\Big(W\Big(\frac{r_m+1}{b^m}\Big)-W\Big(\frac{r_m}{b^m}\Big)\Big)-\sum_{m=1}^n \alpha^{-m}\Delta_m\kappa W(1)\nonumber\\
&=\int_0^1f_n(t)\,dW(t),\label{fractional bridge thm eq 1}
\end{align}
Here we have used the fact that the Wiener integral of the step function  $f_n$ with respect to the fractional Brownian motion $W$ is given by the standard Riemann-type sum
(see, e.g., p.~16 in~\cite{Mishura}). 

 Since $\kappa$ is H\"older continuous with an exponent strictly larger than $H$, our assumption $\alpha b^H>1$ implies that the series $\sum_{m=1}^\infty \alpha^{-m}|\Delta_m\kappa|$ is finite. Moreover, for $k,\,n\in\bN$ with $k<n$ and $p=1/H$,
 \begin{equation}\label{fH in L1H eq}
 \|f_n-f_k\|_{L^{p}[0,1]}\le \sum_{m=k+1}^n\alpha^{-m}b^{-mH}+ \sum_{m=k+1}^n \alpha^{-m}|\Delta_m\kappa|,
 \end{equation}
 and the right-hand side can be made arbitrarily small by making $k$ large. 
 Therefore, the sequence $(f_n)$ converges in $L^{p}[0,1]$ to 
\begin{equation}\label{fnHtofH in L1H eq}
f:=\sum_{m=1}^\infty\alpha^{-m}\Ind{[\frac{r_m}{b^m},\frac{r_m+1}{b^m}]} -\sum_{m=1}^\infty \alpha^{-m}\Delta_m\kappa.
\end{equation}
We claim next that there exists a nonempty open interval on which $f$ is nonzero. Indeed, we have
\begin{equation}\label{fH>0 eq}
 f(x)\ge \alpha^{-n}-\sum_{m=1}^\infty \alpha^{-m}|\Delta_m\kappa|\qquad\text{for }x\in\Big[\frac{r_n}{b^n},\frac{r_n+1}{b^n}\Big],
\end{equation}
 and when $n$ is sufficiently large, the right-hand side of the preceding inequality will be strictly positive.

Our next goal is to show that the Wiener integral $\int_0^1 f(t)\,dW(t)$ exists and that
\begin{equation}\label{fnHtofH eq}
\sum_{m=1}^n\alpha^{-m}\Big(B\Big(\frac{r_m+1}{b^m}\Big)-B\Big(\frac{r_m}{b^m}\Big)\Big)=\int_0^1f_n(t)\,dW(t)\longrightarrow\int_0^1 f(t)\,dW(t)\quad\text{in $L^2$.}
\end{equation}
To this end, we will separately consider the cases $H=1/2$, $1/2<H<1$, and $0<H<1/2$.

 First, we consider the case $H=1/2$. In this case, $W$ is a standard Brownian motion, and so  \eqref{fractional bridge thm eq 1}, \eqref{fnHtofH in L1H eq}, and the standard It\^o isometry yield our claim. In particular,
 \begin{equation}\label{Ito isometry aux eq}
 \bE_W\bigg[\bigg(\sum_{m=1}^\infty\alpha^{-m}\Big(B\Big(\frac{r_m+1}{b^m}\Big)-B\Big(\frac{r_m}{b^m}\Big)\Big)\bigg)^2\bigg]=\mathbb E_W\bigg[\bigg(\int_0^1 f(t)\,dW(t)\bigg)^2\bigg]=\| f\|_{L^2[0,1]}^2,
 \end{equation}
where the right-hand side is finite and strictly positive.

Next, we consider the case $1/2<H<1$. It follows from \eqref{fnHtofH in L1H eq} in conjunction with Theorem~1.9.1~(ii) and Equation (1.6.3) in~\cite{Mishura} that $\int_0^1 f(t)\,dW(t)$ exists and \eqref{fnHtofH eq} holds. In particular, the first identity in \eqref{Ito isometry aux eq} holds also in our current case $1/2<H<1$, and another application of Theorem~1.9.1~(ii) in~\cite{Mishura} yields that the expectation \eqref{sum B rm eq} is finite. To show that it is also strictly positive, 
we combine \eqref{fnHtofH eq} and
the fact that $ f\in L^{p}[0,1]$ with Lemma~1.6.6, Theorem~1.9.1~(ii), and Equation~(1.6.14) in~\cite{Mishura} so as to obtain a constant $C>0$ such that
\begin{align}
\bE_W\bigg[\bigg(\sum_{m=1}^\infty\alpha^{-m}\Big(B\Big(\frac{r_m+1}{b^m}\Big)-B\Big(\frac{r_m}{b^m}\Big)\Big)\bigg)^2\bigg]
&=C \int_0^1\int_0^1 f(u) f(v)|u-v|^{2H-2}\,du\,dv.\label{fH kernel int eq}
\end{align}
 Since $H<1$, the function $t\mapsto t^{2H-2}$ is strictly convex and decreasing on $(0,\infty)$. The integral kernel $K(u,v):=|u-v|^{2H-2}$  is hence strictly positive definite; see Proposition 2 in  \cite{ASS}. Since we have already seen that $ f(t)$ is nonzero on a nonempty open interval, the integral on the right-hand side of  
\eqref{fH kernel int eq} must hence be strictly positive.

Finally, we consider the case $0<H<1/2$. 
Recall from Section 1.6 in~\cite{Mishura} that there exists an unbounded linear integral operator $M_-^H$ from $L^1[0,1]$ to $ L^2(\mathbb R)$ such that
\begin{equation}\label{Wiener integral isometry eq}
\mathbb E\bigg[\bigg(\int_0^1 g(t)\,dW(t)\bigg)^2\bigg]=\|M_-^Hg\|^2_{L^2(\mathbb R)}
\end{equation}
for all $g$ in the domain of $M_-^H$, which is denoted by   $L^2_H[0,1]$. The specific form of $M_-^H$ will not be needed here. All we will need is that there exists a universal constant $C_H>0$ such that 
\begin{equation*}
\big|\widehat{M_-^Hg}(x)\big|=C_H|\widehat g(x)|\cdot|x|^{\frac12-H},\qquad x\in\mathbb R,
\end{equation*}
where $\widehat g(x)=\int e^{ixt}g(t)\,dt$ denotes the Fourier transform of  $g$; see Theorem 1.1.5 and (1.3.3) in~\cite{Mishura}. By Parseval's identity, we hence have 
\begin{equation}\label{MH Fourier trafo eq}
\|M_-^Hg\|_{L^2(\mathbb R)}^2=C_H^2\int |\widehat g(x)|^2\cdot|x|^{1-2H}\,dx,\qquad g\in L^2_H[0,1].
\end{equation}
Note that
$$\big|\widehat{\Ind{[\frac{r_m}{b^m},\frac{r_m+1}{b^m}]}}(x)\big|=\Big|\frac{e^{ix(r_m+1)/b^m}-e^{ixr_m/b^m}}{x}\Big|\le\begin{cases}b^{-m}&\text{if $|x|\le b^m$,}\\
2/|x|&\text{otherwise.}
\end{cases}
$$
Therefore, for $n< N$ and $c:=\|M^H_-\Ind{[0,1]}\|_{L^2(\mathbb R)}/C_H$, 
\begin{align}
\lefteqn{\frac1{C_H}\big\|M^H_-f_N-M^H_-f_n\big\|_{L^2(\mathbb R)}}\nonumber\\
&\le\frac1{C_H}\sum_{m=n+1}^N \alpha^{-m}\big\|M^H_-\Ind{[\frac{r_m}{b^m},\frac{r_m+1}{b^m}]}\big\|_{L^2(\mathbb R)}+c\sum_{m=n+1}^N \alpha^{-m}|\Delta_m\kappa|\nonumber\\
&\le \sum_{m=n+1}^N \alpha^{-m}b^{-m}\sqrt{2\int_{0}^{b^m}x^{1-2H}\,dx}+ \sum_{m=n+1}^N \alpha^{-m}\sqrt{8\int_{b^m}^\infty x^{-1-2H}\,dx}+c\sum_{m=n+1}^N \alpha^{-m}|\Delta_m\kappa|\nonumber\\
&= \frac1{\sqrt{1-H}}\sum_{m=n+1}^N \alpha^{-m}b^{-mH}+\sqrt{\frac4H}\sum_{m=n+1}^N \alpha^{-m}b^{-mH}+c\sum_{m=n+1}^N \alpha^{-m}|\Delta_m\kappa|.\nonumber
\end{align}
Since $\alpha b^H>1$, the latter expression is less than any given $\eps>0$ as soon as $n$ is sufficiently large.
By Remark 1.6.3 in~\cite{Mishura}, the space $L^2_H[0,1]$ is complete with respect to the norm $\|f\|_{L^2_H[0,1]}=\|M_-^Hf\|_{L^2(\mathbb R)}$ if $0<H<1/2$, and so we must have $f_n\to f$ in $L^2_H[0,1]$. Thus, 
the Wiener integral of $ f$ exists and we also have  \eqref{fnHtofH eq}, which gives
\begin{equation}\label{sum Wiener int eq}
\int_0^1 f(t)\,dW(t)=\sum_{m=1}^\infty\alpha^{-m}\Big(B\Big(\frac{r_m+1}{b^m}\Big)-B\Big(\frac{r_m}{b^m}\Big)\Big).
\end{equation}
By \eqref{Wiener integral isometry eq}, the $L^2$-norm of this Wiener integral is given by $\|M_-^H f\|_{L^2(\mathbb R)}$, and \eqref{MH Fourier trafo eq} yields that $\|M_-^H f\|_{L^2(\mathbb R)}$  is finite and also strictly positive, since $ f$ is nonzero on a nonempty open interval. 
\end{proof}

\begin{proof}[Proof of \Cref{main thm} {\rm (a)}] To apply \Cref{p variation prop} in \Cref{Hoelder section}, choose $\gamma\in(K,H)$ so that $\alpha b^\gamma>1$ and pick a version of $B$ with $\gamma$-H\"{o}lder continuous sample paths.  \Cref{fractional bridge thm lemma}
 yields that 
$$0<\bE_W\bigg[\bigg(\sum_{m=1}^\infty\alpha^{-m}\Big(B\Big(\frac{R_m+1}{b^m}\Big)-B\Big(\frac{R_m}{b^m}\Big)\Big)\bigg)^2\bigg]<\infty\quad\text{$\mathbb P_R$-a.s.}$$
The argument of the expectation is normally distributed, and so 
$$\mathbb P_W\bigg(\sum_{m=1}^\infty\alpha^{-m}\Big(B\Big(\frac{R_m+1}{b^m}\Big)-B\Big(\frac{R_m}{b^m}\Big)\Big)=0\,\bigg)=0\quad\text{$\mathbb P_R$-a.s.} $$
Hence we conclude that
$$
\mathbb P\bigg(\sum_{m=1}^\infty\alpha^{-m}\Big(B\Big(\frac{R_m+1}{b^m}\Big)-B\Big(\frac{R_m}{b^m}\Big)\Big)=0\bigg)=0.
$$
Therefore, \Cref{p variation prop} yields the result.
\end{proof}

\begin{remark}\label{pth var rep remark}
It follows from \Cref{p variation prop} that the $p$-th variation  of $X$ is $\bP_W$-a.s.~given by
\begin{equation}\label{pth var rep remark eq}
\<X\>^{(p)}_1=\bE_R\bigg[\bigg|\sum_{m=1}^\infty\alpha^{-m}\Big(B\Big(\frac{R_m+1}{b^m}\Big)-B\Big(\frac{R_m}{b^m}\Big)\Big)\bigg|^p\bigg].
\end{equation}
Furthermore,  for each realization  of the random variables $(R_m)$, the identity  \eqref{sum Wiener int eq}  shows that  
the expression
$$\sum_{m=1}^\infty\alpha^{-m}\Big(B\Big(\frac{R_m+1}{b^m}\Big)-B\Big(\frac{R_m}{b^m}\Big)\Big)
$$
is equal to a Wiener integral with respect to the fractional Brownian motion $W$. Thus, the right-hand side of \eqref{pth var rep remark eq} can be regarded as a mixture of the $p$-th powers of certain Wiener integrals.
\end{remark}

\subsection{Proof of \Cref{main thm} (b)}\label{thm (c) proof section}

In this section, we will prove part (b) of \Cref{main thm}. Recall that this part addresses the situation in which $K>H$, which is equivalent to either $\alpha b^H<1$ or $\alpha^p b<1$. 
 Let us introduce the notation
\begin{equation}\label{Vn eq}
V_n:=\sum_{k=0}^{b^n-1}|X((k+1)b^{-n})-X(kb^{-n})|^p.
\end{equation}
Recall also the three probability spaces $(\Omega_R,\mathscr{F}_R,\mathbb{P}_R)$,  $(\Omega_W,\mathscr{F}_W,\mathbb{P}_W)$, and  $(\Omega,\mathscr{F},\mathbb{P})=(\Omega_W\times \Omega_R,\cF_W\otimes\cF_R,\bP_W\otimes\bP_R)$ introduced above. Sections \ref{Convergence in expectation section} and \ref{conce} are devoted to the proof of \Cref{main thm} (b) when $t=1$, and we extend the arguments for a general $t\in[0,1]$ in Section \ref{linear section}.

\subsubsection{Convergence of the expectation}
\label{Convergence in expectation section}

In this first part of the proof, we will prove that the expected $p$-th variation, $\bE_W[V_n]$, converges to $c_H/(1-\alpha^2b^{2H})^{p/2}$, where
\begin{equation*}
c_H:=\frac{2^{1/(2H)}\Gamma(\frac{H+1}{2H})}{\sqrt{\pi}}.
\end{equation*}
Here and later, we denote by $\bone_A$ the indicator function of a set $A$.

\begin{lemma}\label{l1}
The expected $p$-th variation  is of the form
$$
\bE_W[V_n]=(\alpha^pb)^n\bE\bigg[\bigg|\int_0^1f_n(x)\,dW(x)\bigg|^p\bigg],
$$
where $\bE[\,\cdot\,]$ denotes the expectation with respect to $\bP=\bP_W\otimes\bP_R$ and $f_n(x)=g_n(x)-h_n$ for
 $$g_n(x):=\sum_{m=1}^n\alpha^{-m}\bone_{[R_mb^{-m},(R_m+1)b^{-m}]}(x)\quad\text{and}\quad
h_n:=\sum_{m=1}^n\alpha^{-m}(\kk((R_m+1)b^{-m})-\kk(R_mb^{-m})).$$
\end{lemma}

\begin{proof}
By \eqref{Vn Rm eq} and \eqref{fbb},
\begin{align*}
\bE[V_n]&=(\alpha^pb)^n\bE_W\bigg[\bE_R\bigg[\bigg|\sum_{m=1}^n\alpha^{-m}\bigg(B((R_m+1)b^{-m})-B(R_mb^{-m})\bigg)\bigg|^p\bigg]\bigg]\\
&=(\alpha^pb)^n\bE_R\bigg[\bE_W\bigg[\bigg|\int_0^1f_n(x)\,dW(x)\bigg|^p\bigg]\bigg],
\end{align*}
where we have  used the fact that the Wiener integral of the step function  $f_n$ is given by the standard Riemann-type sum
(see, e.g., p.~16 in~\cite{Mishura}). 
\end{proof}

We will eventually prove that 
$$\bE\bigg[\bigg|\int_0^1f_n(x)\,dW(x)\bigg|^p\bigg]$$ 
is of order $(\alpha^pb)^{-n}$ and the contribution from the time-independent  constants $h_n$ are asymptotically negligible. This argument will  become valid after having shown that the contribution from the functions $g_n$  gives the correct magnitude. This is our current focus. We start with the following elementary lemma.

\begin{lemma}\label{l2}
For given $\gamma\in(0,1)$, $C>0$, and $0<a<\beta$, suppose that we have constants $C_n=(1+o(1))C\beta^n$ and random variables $\eta_n$ for which $\bE[|\eta_n|]=O(a^n)$. Then $\bE[|C_n+\eta_n|^\gamma]=(1+o(1))C^\gamma \beta^{n\gamma}$.

\end{lemma}

\begin{proof}
We have
\begin{align}
\bE\bigg[\bigg|1+\frac{\eta_n}{C_n}\bigg|^\gamma\bigg]&=\int_0^\infty \bP\bigg(\bigg|1+\frac{\eta_n}{C_n}\bigg|^\gamma>x\bigg)\,dx\nonumber\\
&=\int_0^\infty \bP\bigg(\frac{\eta_n}{C_n}>x^{1/\gamma}-1\bigg)\,dx+\int_0^\infty \bP\bigg(\frac{\eta_n}{C_n}<-x^{1/\gamma}-1\bigg)\,dx.\label{l2 proof eq}
\end{align}
We show that the first integral converges to one. First, we have for any $\ee>0$, 
\begin{align*}
1\geq\int_0^1 \bP\bigg(\frac{\eta_n}{C_n}>x^{1/\gamma}-1\bigg)\,dx\geq \int_0^{1-\ee}\bP\big(|\eta_n|<(1-(1-\ee)^{1/\gamma})C_n\big)\,dx\lra 1-\ee.
\end{align*}
To deal with the remaining part of the integral, we let $L$ denote a constant for which $\bE[|\eta_n|]/C_n\le L a^n/\beta^n$. Then, by Markov's inequality for $y>0$,
$$
\bP(|\eta_n|>y C_n)\leq \frac{\bE[|\eta_n|]}{y C_n}\leq L\frac{a^n}{y\beta^n}.
$$
Therefore,
\begin{align*}
0\leq \int_1^\infty \bP\bigg(\frac{\eta_n}{C_n}>x^{1/\gamma}-1\bigg)\,dx\leq \ee+L\frac{a^n}{\beta^n}\int_{1+\ee}^\infty\frac1{x^{1/\gamma}-1}\,dx\lra \ee.
\end{align*}
The second integral in \eqref{l2 proof eq} converges to zero by a similar argument. This completes the proof.
\end{proof}

In the following, we will frequently deal with sums of covariances, for which the following easy observation will be useful.

\begin{lemma}
Let $m,k\in\bN,m>k$. Let $S=\{0,1,\dots,b^m-1\}$ and fix $j\in \{1-b^m,\dots,0,\dots,b^m-1\}$. Then
$
\#\{i\in S:b^m\fp{ib^{-k}}-i=j\}\leq 1$.
\label{ijs}
\end{lemma}

\begin{proof} Note that $b^m\fp{ib^{-k}}=(i b^{m-k})\!\!\mod b^m$, where $n\!\!\mod{b^m}$ is the remainder of $n$ divided by $b^m$.
Suppose $i_1,i_2\in S$ are such that $(i_1b^{m-k})\!\!\mod{b^m}-i_1=(i_2b^{m-k})\!\!\mod{b^m}-i_2$. Then $b^m$ divides $(i_1-i_2)(b^{m-k}-1)$ and  $b^m$ divides $(i_1-i_2)$, implying $i_1=i_2$. 
\end{proof}

From now on, $L>0$ will denote a generic constant that may depend only on $\alpha,b,H,\kappa$ but not on anything else (in particular, not on $n,m,k$). The value of $L$ may change at each occurrence.

\begin{proposition}\label{p1}
For $g_n$ as in \Cref{l1}, we have
$$(\alpha^pb)^n\bE\bigg[\bigg|\int_0^1g_n(x)\,dW(x)\bigg|^p\bigg]=\frac{(1+o(1))c_H}{(1-\alpha^2b^{2H})^{p/2}}.$$
\end{proposition}

\begin{proof}If $Y$ is any centered Gaussian random variable, then $\bE[|Y|^p]=\bE[Y^2]^{p/2}\bE[|Z|^p]$, where $Z$ is standard normally distributed. Since, moreover, $\bE[|Z|^p]=c_H$, we have
\begin{equation}\label{expectation factor eq}
\bE\bigg[\bigg|\int_0^1g_n(x)\,dW(x)\bigg|^p\bigg]=c_H\bE_R\bigg[\bigg(\bE_W\bigg[\bigg(\int_0^1g_n(x)\,dW(x)\bigg)^2\bigg]\bigg)^{p/2}\bigg].
\end{equation}
To deal with the $\cF_R$-measurable random variable $\bE_W[(\int_0^1g_n(x)\,dW(x))^2]$, we define 
\begin{align*}
\xi_{k,m}&:=\bE_W\Big[\big(W((R_k+1)b^{-k})-W(R_kb^{-k})\big)\big(W((R_m+1)b^{-m})-W(R_mb^{-m})\big)\Big].
\end{align*}
With this notation, the definition of $g_n$ gives
\begin{align}\label{d0}
\bE_W\bigg[\bigg(\int_0^1g_n(x)\,dW(x)\bigg)^2\bigg]=\sum_{m=1}^n\sum_{k=1}^n\alpha^{-m-k}\xi_{k,m}.
\end{align}
Note that the diagonal terms $\xi_{m,m}$ are deterministic and given by $b^{-2mH}$. Hence,
\begin{align}C_n:=\sum_{m=1}^n\xi_{m,m}=\sum_{m=1}^n\alpha^{-2m}b^{-2mH}=
\frac{(\alpha^2b^{2H})^{-n}-1}{1-\alpha^2b^{2H}}=(1+o(1))\frac{(\alpha b^{H})^{-2n}}{1-\alpha^2b^{2H}}.\label{d1}
\end{align}
When denoting the sum of all off-diagonal terms by
\begin{equation*}
\eta_n:=2\sum_{1\leq k<m\leq n}\alpha^{-m-k}\xi_{k,m},
\end{equation*}
we get that
\begin{align}\label{d1a}
\bE\bigg[\bigg|\int_0^1g_n(x)\,dW(x)\bigg|^p\bigg]=c_H\bE_R\big[\big|C_n+\eta_n\big|^{p/2}\big].
\end{align}
For dealing with the right-hand expectation, we need to distinguish between the cases $H\ge1/2$ and $H<1/2$.

For $H\ge 1/2$, the increments of fractional Brownian motion are nonnegatively correlated, so each $\xi_{k,m}$ is nonnegative, and so is $\eta_n$. 
To apply \Cref{l2}, we aim to bound $\bE[\xi_n]=\bE[|\xi_n|]$ from above. 
Due to the stationarity of increments of $W$, the random variable $\xi_{k,m}$ depends only on $R_kb^{-k}-R_mb^{-m}$. Note that since $k<m$, we have $R_kb^{-k}=\fp{R_mb^{-k}}$. So using the fact that $R_m$ is uniformly distributed on $\{0,1,\dots, b^m-1\}$ and  applying \Cref{ijs}  and a telescoping argument yields
\begin{align}
\bE_R[\xi_{k,m}]&=b^{-m}\sum_{i=0}^{b^m-1}\bE_W\big[(W(\fp{ib^{-k}}+b^{-k})-W(\fp{ib^{-k}}))(W((i+1)b^{-m})-W(ib^{-m}))\big]\nonumber\\
&=b^{-m}\sum_{i=0}^{b^m-1}\bE_W\big[W(b^{-k})(W((i+1)b^{-m}-\fp{ib^{-k}})-W(ib^{-m}-\fp{ib^{-k}}))\big]\nonumber\\
&\leq b^{-m}\sum_{j=-b^m}^{b^m-1}\bE_W\big[W(b^{-k})(W((j+1)b^{-m})-W(jb^{-m}))\big]\label{tel}\\
&=b^{-m}\bE_W\big[W(b^{-k})(W(1)-W(-1))\big]=b^{-m}((1+b^{-k})^{2H}-(1-b^{-k})^{2H})\nonumber\\
&\leq Lb^{-m-k},\nonumber
\end{align}where the last step follows from the mean-value theorem. Combining the above yields
\begin{equation}\label{etan expectation eq}
\bE_R[|\eta_n|]=\bE_R[\eta_n]=2\sum_{1\leq k<m\leq n}\alpha^{-m-k}\bE_R[\xi_{k,m}]\leq L\sum_{1\leq k<m\leq n}\alpha^{-m-k}b^{-m-k}=O(a^{n}),
\end{equation}
where $a=(\alpha b)^{-2}$ for $\alpha b<1$, and $a$ is an arbitrary number in $ (1,(\alpha b^H)^{-2})$ for $\alpha b\geq 1$.
Since $a<(\alpha^2b^{2H})^{-1}=:\beta$, it follows from \Cref{l2} with $C=(1-\alpha^2b^{2H})^{-1}$,  and $\gamma=p/2$ that  $\bE[|C_n+\eta_n|^{p/2}]=(1+o(1))C^{p/2} (\alpha^pb)^{-n}$. In view of \eqref{d1a}, this completes the proof for $1/2\le H<1$.

We now turn to the case $0<H< 1/2$, for which  $p/2\ge1$. We will  prove below that for $k<m$,
\begin{equation}\label{newl eq}
\Vert \xi_{m,k}\Vert_{L^{p/2}(\bP_R)}\leq L(b^{-m}+b^{-4Hm}).
\end{equation}
Then Minkowski's inequality and $0<H<1/2$ yields that
\begin{align*}
\Vert \eta_n\Vert_{L^{p/2}(\bP_R)}&\leq 2\sum_{1\leq k<m\leq n}\alpha^{-m-k}\Vert \xi_{k,m}\Vert_{L^{p/2}(\bP_R)}\leq L\sum_{1\leq k<m\leq n}\alpha^{-m-k}(b^{-m}+b^{-4Hm})=o((\alpha^2 b^{2H})^{-n}).
\end{align*}
Thus, by
 \eqref{d1a}  and another application of Minkowski's inequality,\begin{align*}\bE\bigg[\bigg|\int_0^1g_n(x)\,dW(x)\bigg|^p\bigg]^{2/p}&=c_H^{2/p}\big(C_n+O(\Vert \eta_n\Vert_{L^{p/2}(\bP_R)})\big)=c_H^{2/p}C_n+o((\alpha^2 b^{2H})^{-n}).
\end{align*}
Multiplying the preceding identity with $(\alpha^pb)^{2n/p}=(\alpha^2b^{2H})^n$ and using \eqref{d1} will then yield the assertion for $0<H< 1/2$.

It remains to establish \eqref{newl eq}. By arguing as in \eqref{tel}, we write
\begin{align*}
\lefteqn{\bE_R\big[|\xi_{m,k}|^{p/2}\big]}\\
&\leq b^{-m}\sum_{j=-b^{m}}^{b^m-1}\big|\bE_W\big[W(b^{-k})(W((j+1)b^{-m})-W(jb^{-m}))\big]\big|^{p/2}\\
&=b^{-m}\sum_{j=-b^{m}}^{b^m-1}\left|\big|(j+1)b^{-m}\big|^{2H}+\big|jb^{-m}-b^{-k}\big|^{2H}-\big|jb^{-m}\big|^{2H}-\big|(j+1)b^{-m}-b^{-k}\big|^{2H}\right|^{p/2}\\
&\leq Lb^{-m}\sum_{j=-b^{m}}^{b^m-1}\left(\left|\big|(j+1)b^{-m}\big|^{2H}-\big|jb^{-m}\big|^{2H}\right|^{p/2}+\left|\big|jb^{-m}-b^{-k}\big|^{2H}-\big|(j+1)b^{-m}-b^{-k}\big|^{2H}\right|^{p/2}\right)\\
&\leq Lb^{-m}\sum_{j=0}^{2b^m-1}\left(\big((j+1)b^{-m}\big)^{2H}-\big(jb^{-m}\big)^{2H}\right)^{p/2}\\
&\leq Lb^{-2m}+Lb^{-m}\sum_{j=2}^{2b^m-1}\left(\big((j+1)b^{-m}\big)^{2H}-\big(jb^{-m}\big)^{2H}\right)^{p/2}.
\end{align*}
By the mean-value theorem, for $j\geq 1$,
$$
\big((j+1)b^{-m}\big)^{2H}-\big(jb^{-m}\big)^{2H}\leq L\big(jb^{-m}\big)^{2H-1}b^{-m}=Lj^{2H-1}b^{-2mH}.
$$
Therefore, since $p>2$,
$$
b^{-m}\sum_{j=2}^{2b^m-1}\left(\big((j+1)b^{-m}\big)^{2H}-\big(jb^{-m}\big)^{2H}\right)^{p/2}\leq Lb^{-2m}\sum_{j=2}^{2b^m-1}j^{1-p/2}\leq Lb^{-pm/2}.
$$
Combining the above we obtain
$
\Vert \xi_{m,k}\Vert_{L^{p/2}(\bP_R)}\leq L(b^{-2m}+b^{-pm/2})^{2/p}\leq L(b^{-4Hm}+b^{-m})$ and thus \eqref{newl eq}.
\end{proof}

\begin{proposition}\label{c1}
\Cref{p1} holds with $g_n$ replaced by $f_n$, i.e.,
$$(\alpha^pb)^n\bE\bigg[\bigg|\int_0^1f_n(x)\,dW(x)\bigg|^p\bigg]=\frac{(1+o(1))c_H}{(1-\alpha^2b^{2H})^{p/2}}.$$
\end{proposition}
\begin{proof}
Recall from \Cref{l1} that $f_n=g_n-h_n$ where $h_n=\sum_{m=1}^n\alpha^{-m}(\kk((R_m+1)b^{-m})-\kk(R_mb^{-m}))$. By \eqref{expectation factor eq}, 
Minkowski's inequality and the $\tau$-H\"older continuity of $\kappa$, 
\begin{align*}\bigg\Vert \int_0^1h_n\,dW(x)\bigg\Vert_{L^p(\bP)}&=\Big(\bE_R\big[|h_n|^p\cdot\bE_W[|W(1)|^{p}]\big]\Big)^{1/p}\\
&=c_H^{1/p}\n{\sum_{m=1}^n\alpha^{-m}(\kk((R_m+1)b^{-m})-\kk(R_mb^{-m}))}_{L^p(\bP_R)}\\
&\leq c_H^{1/p}\sum_{m=1}^n\alpha^{-m}\n{(\kk((R_m+1)b^{-m})-\kk(R_mb^{-m}))}_{L^p(\bP_R)}\\
&\leq L\sum_{m=1}^n\alpha^{-m}b^{-m\tau}=o((\alpha b^H)^{-n}).
\end{align*}
Thus the assertion follows by applying Minkowski's inequality to $\Vert \int f_n\,dW\Vert_{L^p(\bP)}=\Vert \int (g_n-h_n)\,dW\Vert_{L^p(\bP)}$ and using \Cref{p1}. 
\end{proof}

\subsubsection{A concentration bound}\label{conce}

Having proved the convergence of the expected $p$-th variation, we now turn to the second part of the proof, which establishes a  concentration inequality and thus $\bP_W$-a.s.~convergence. Recall the notation $V_n$ from \eqref{Vn eq}. We also introduce the shorthand notation 
$$t^{(n)}_i=ib^{-n},\qquad \text{$n\in\bN$ and $i=0,\dots, b^n$,}
$$
and we will simply write $t_i$ in place of $t_i^{(n)}$ if the value of $n$ is clear from the context.
The following lemma can be proved analogously as  Lemma 10.2.2 in~\cite{MarcusRosen}; all one needs is to replace \cite[Equation (5.152)]{MarcusRosen} with the Borell--TIS inequality in the form of Theorem 2.1.1 in \cite{AdlerTaylor}.

\begin{lemma}\label{Marcus Rosen lemma}
Let $q>1$ with $p^{-1}+q^{-1}=1$ and define
\begin{equation*}
M_n:=\bigg\{(\mu_1,\dots,\mu_{b^n})\in\bR^{b^n}:\sum_{j=1}^{b^n}|\mu_j|^q\le1\bigg\}
\end{equation*}
and 
\begin{align}
\sigma_n^2:=\sup_{(\mu_k)\in M_n}\sum_{i=1}^{b^n}\sum_{j=1}^{b^n}\mu_i\mu_j\bE[(X(t_i)-X(t_{i-1}))(X(t_j)-X(t_{j-1}))].\label{s}
\end{align}
Then
\begin{align}\label{ci1}
\bP\left(|V_n^H-\bE[V_n^H]|>s\right)\leq 2e^{-\frac{s^2}{2\sigma_n^2}}.
\end{align}
\end{lemma}

The preceding lemma will be needed in the proof of the following proposition. In the sequel,  $\lambda$ will denote a generic constant in $(0,1)$ that may depend only on $\alpha,b,H$ and that may differ from occurrence to occurrence.

\begin{proposition}
Suppose that \eqref{ci1} holds and $\sigma_n=O(\lambda^n)$ for some $\lambda\in(0,1)$. Then $V_{n}$ converges almost surely to
$$C_p:=\frac{c_H}{(1-\alpha^2b^{2/p})^{p/2}}=\frac{c_H}{(1-\alpha^2b^{2H})^{1/(2H)}}.$$ That is, \Cref{main thm} holds for $t=1$.\label{l6}
\end{proposition}

\begin{proof}
Combining \Cref{c1} and \Cref{l1} yields 
$
\lim_{n}\bE[V_{n}]= C_p$. We also claim that the sequence $(V_n)_{n\in\bN}$ is uniformly integrable. To see why, choose $n_0\in\bN$ such that $\bE[V_n]\le C_p+1$ and $\sigma_n\le \sqrt{2\pi}$ for all $n\ge n_0$. Then, for $n\ge n_0$ and $c>(C_p+2)^p$,
\begin{align*}
\bE[V_n\Ind{\{V_n>c\}}]&=\int_c^\infty\bP(V_n>r)\,dr=p\int_{c^{1/p}-C_p-2}\bP(V_n>(C_p+2+s)^p)(C_p+2+s)^{p-1}\,ds\\
&\le p\int_{c^{1/p}-C_p-2}^\infty \bP\left(|V_n^H-\bE[V_n^H]|>s+\frac{\sigma_n}{\sqrt{2\pi}}\right)(C_p+2+s)^{p-1}\,ds\\
&\le 2p\int_{c^{1/p}-C_p-2}^\infty e^{-s^2/4\pi}(C_p+2+s)^{p-1}\,ds,
\end{align*}
where we have used \eqref{ci1} in the final step. Clearly, the latter integral can be made arbitrarily small by increasing  $c$, which proves the claimed uniform integrability.

 Next, since $\bE[V_{n}^H]\leq \bE[V_{n}]^H$, the sequence $(\bE[V_{n}^H])_{n\in\bN}$ is bounded. Suppose there is a subsequence $(n_k)$ such that $\bE[V_{n_k}^H]$ converges to the finite limit $\ell$  as $k\ua\infty$. Then \eqref{ci1} with the choice {$s_n=n\sigma_n$} and the  Borel--Cantelli lemma  give $V_{n_k}^H\to \ell$ $\bP$-a.s.~and in turn $V_{n_k}\to \ell^p$ $\bP$-a.s. Due to the established uniform integrability, the latter convergence also holds in $L^1$, and we obtain $\ell^p=C_p$. It follows that $\ell$ is the unique accumulation point of the sequence $(\bE[V^H_n])_{n\in\bN}$  and  equal to $C_p^{1/p}$. Therefore, we can replace the above subsequence $(n_k)$ by $\bN$, so that $V_{n}\to C_p$ $\bP$-a.s.~as required. 
\end{proof}

In the remainder of this section, we prove that $\sigma_n^2=O(\lambda^n)$ for some $\lambda\in(0,1)$. The first obvious step is to plug \eqref{WW def eq} into \eqref{s}. Fixing $n\in\bN$ and using the shorthand notation
\begin{equation}\label{rhomkij}
\rho^{(m,k)}_{i,j}:=\bE[(B(\fp{b^m t_i})-B(\fp{b^m t_{i-1}}))(B(\fp{b^k t_j})-B(\fp{b^k t_{j-1}}))],
\end{equation}
this gives
\begin{align}
\sigma_n^2&=\sup_{(\mu_k)\in M_n}\sum_{i=1}^{b^n}\sum_{j=1}^{b^n}\mu_i\mu_j\sum_{m=0}^{n-1}\sum_{k=0}^{n-1}\alpha^{m+k}\rho^{(m,k)}_{i,j}=\sup_{(\mu_k)\in M_n}\sum_{m=0}^{n-1}\sum_{k=0}^{n-1}\alpha^{m+k}\sum_{i=1}^{b^n}\sum_{j=1}^{b^n}\mu_i\mu_j\rho^{(m,k)}_{i,j}\nonumber\\
&\leq \sup_{(\mu_i)\in M_n}\sup_{(\nu_j)\in M_n}\sum_{m=0}^{n-1}\sum_{k=0}^{n-1}\alpha^{m+k}\sum_{i=1}^{b^n}\sum_{j=1}^{b^n}\mu_i\nu_j\rho^{(m,k)}_{i,j}.\label{s2}
\end{align}

\begin{lemma}\label{cov}
Let $1/2\le H<1$ and consider two disjoint intervals of lengths $b^{-m},b^{-k}$ in $[0,1]$ that are apart by the distance $\delta$. Then the covariance of the increments of $W$ on these two intervals is bounded by $Lb^{-m-k}\delta^{2H-2}$.
\end{lemma}

\begin{proof}
We assume that $m>k$ and the two intervals are denoted $[u,v],[s,t]$ with $0\leq u<v=u+b^{-m}<s<t=s+b^{-k}\leq 1$. The proofs for the other cases are analogous. 

Since $H\ge 1/2$, the function $x\mapsto x^{2H}$ is convex and its derivative is bounded by $L$ on $[0,1]$. We also record here the standard fact that
\begin{align}
\bE_W[(W(v)-W(u))(W(t)-W(s))]=\frac{1}{2}(|v-s|^{2H}+|u-t|^{2H}-|v-t|^{2H}-|u-s|^{2H}).\label{cov!}
\end{align}Observe that $\delta=s-v<s-u<t-v<t-u$. By the mean-value theorem, there are $x_1\in(s-v,s-u),x_2\in(t-v,t-u)$ such that
\begin{align*}
\bE_W[(W(v)-W(u))(W(t)-W(s))]&=Lb^{-m}(x_2^{2H-1}-x_1^{2H-1})\nonumber\\
&\leq Lb^{-m}(b^{-k}+b^{-m})\delta^{2H-2}\leq Lb^{-k-m}\delta^{2H-2},
\end{align*}completing the proof.
\end{proof}

For the case $0<H< 1/2$, the following lemma, in a similar sense as \Cref{cov}, gives estimates of covariances of increments of $W$ that are sufficiently apart.

\begin{lemma}\label{l5}
For $0<H\leq 1/2$ there exists a constant $L>0$ such that for all $1\leq i\leq b^n$,
\begin{align*}
\sum_{j=1\atop |i-j|>2}^{b^n}\Big|\bE_W\big[(W(t_i)-W(t_{i-1}))(W(t_j)-W(t_{j-1}))\big]\Big|\leq Lb^{-2nH}.\label{ij}
\end{align*}
\end{lemma}

\begin{proof}
By the symmetry and stationarity of the increments, it suffices to consider the case $i=1$. That is, it suffices to show that
$$
T_n:=\sum_{j=3}^{b^n}\Big|\bE_W\big[W(t_1)(W(t_j)-W(t_{j-1}))\big]\Big|\leq Lb^{-2nH}.
$$
This obviously holds for $H=1/2$. For $0<H<1/2$, we use \eqref{cov!} and the mean-value theorem to get
\begin{align*}
T_{n}&=\frac{1}{2}\sum_{j=3}^{b^n}b^{-2nH}\left(2(j-1)^{2H}-j^{2H}-(j-2)^{2H}\right)\leq Lb^{-2nH}\sum_{j=3}^{b^n}(j-2)^{2H-2}\le Lb^{-2nH}.
\end{align*}This concludes the proof.
\end{proof}

For a function $f:[0,1]\to\bR$ and $0\le s< t$ with $(s,t)\notin\bN_0\times\bN$, we introduce the notation
\begin{equation}\label{Delta def}
\Delta f(s,t)=\begin{cases}f(1)-f(\fp{s})&\text{if $t\in\bN$,}\\
f(\fp{t})-f(0)&\text{if $s\in\bN_0$,}\\
f(\fp{t})-f(\fp{s})&\text{otherwise.}
\end{cases}
\end{equation}
Then we have the relations
\begin{equation}\label{Delta B Delta W relation}
\Delta B(s,t)=B(\fp{t})-B(\fp{s})\qquad \text{and}\qquad \Delta W(s,t)=\Delta B(s,t)+\Delta \kappa(s,t)W(1).
\end{equation}
So $\rho^{(m,k)}_{i,j}$ from \eqref{rhomkij} has the alternative expression
$$\rho^{(m,k)}_{i,j}=\bE_W[\Delta B(b^m t_{i-1}, b^m t_i)\cdot  \Delta B(b^k t_{j-1},b^k t_j)].
$$
In the same way, we let
\begin{equation*}
\tilde\rho^{(m,k)}_{i,j}:=\bE_W[\Delta W(b^m t_{i-1}, b^m t_i)\cdot  \Delta W(b^k t_{j-1},b^k t_j)].
\end{equation*}
These quantities are well defined as long as $i,j\ge1$ and  $m,k<n$, because then $(b^m t_{i-1}, b^m t_i)$ and $(b^k t_{j-1},b^k t_j)$ do not belong to $\bN_0\times\bN$.

\begin{lemma}\label{l9}
For $H\in(0,1)$, let $h:=(2H)\wedge1$. Then the following inequalities hold.
\begin{enumerate}
\item For $k=0,\dots, n-1$,
\begin{equation}\label{l9 eq1}
\big|\rho^{(0,k)}_{i,j}-\tilde \rho^{(0,k)}_{i,j}\big|\le L\big(b^{(k-2n)\tau}+b^{-n\tau+(k-n)h}+b^{(k-n)\tau-nh}\big).
\end{equation}
\item As $n\ua\infty$, we have for some $\lambda\in(0,1)$ independent of $n$,
\begin{align}\label{lq asser eq}
\sum_{k=0}^{n-1}\alpha^{k}\sup_{(\mu_i)\in M_n}\sup_{(\nu_j)\in M_n}\sum_{i=1}^{b^n}\sum_{j=1}^{b^n}|\mu_i||\nu_j|\big|\rho^{(0,k)}_{i,j}-\tilde \rho^{(0,k)}_{i,j}\big|=O(\lambda^n).
\end{align}
\end{enumerate}

\end{lemma}

\begin{proof} (a) We get from \eqref{Delta B Delta W relation} that $|\rho^{(0,k)}_{i,j}-\tilde \rho^{(0,k)}_{i,j}|\le I+J$, where 
\begin{align*}
I&:=|(\kk(t_i)-\kk(t_{i-1}))\Delta\kk(b^k t_{j-1},b^k t_j)|,\\
J&:=\big|(\kk(t_i)-\kk(t_{i-1}))\bE_W[W(1) \Delta W(b^k t_{j-1},b^k t_j)]\big|+\big|\Delta\kk(b^k t_{j-1},b^k t_j)\bE_W[W(1)(W(t_i)-W(t_{i-1}))]\big|.
\end{align*}
The definition \eqref{Delta def} and the $\tau$-H\"older continuity of $\kappa$ imply that  $|\Delta\kk(b^k t_{j-1},b^k t_j)|\le Lb^{(k-n)\tau}$ and in turn $I\leq Lb^{(k-2n)\tau}$. To deal with $J$, note  that the covariance \eqref{cov!} is H\"older continuous with exponent $h$ in each of its arguments. This gives $J\le L(b^{-n\tau+(k-n)h}+b^{(k-n)\tau-nh})$ and proves (a).
 
 To prove part (b), we note first that for $(\mu_i)\in M_n$, due to H\"older's inequality, 
 \begin{align}\label{Mn Hoelder eq}
  \sum_{i=1}^{b^n}|\mu_i|\le \bigg(\sum_{i=1}^{b^n}|\mu_i|^q\bigg)^{1/q}\bigg(\sum_{i=1}^{b^n}1\bigg)^{1/p}\le b^{nH}.
 \end{align}
For the purpose of this proof, let us denote the expression on the left-hand side of \eqref{lq asser eq}
 by $S_n$ and the right-hand side of \eqref{l9 eq1} by $K_{n,k}$. 
Then \eqref{Mn Hoelder eq} and part (a) yield that
$$S_n\leq L\sum_{k=0}^{n-1}\alpha^{k}\bigg(\sup_{(\mu_i)\in M_n}\sum_{i=1}^{b^n}|\mu_i|\bigg)^2K_{n,k}\leq Lb^{2nH}\sum_{k=0}^{n-1}\alpha^{k}\left(b^{(k-2n)\tau}+b^{-n\tau+(k-n)h}+b^{(k-n)\tau-nh}\right).$$
By evaluating the geometric sum and using $\tau\wedge h>H$, we conclude that $S_n=O(\lambda^n)$ for some $\lambda\in(0,1)$.
\end{proof}

The following basic estimate is a consequence of the above lemmas and serves as the base case for an induction proof.

\begin{lemma}\label{ind1} There exist $\lambda\in(0,1)$ and $L>0$, depending only on $b$, $H$, and $\tau$, such that for all $n$,
\begin{align}
\sup_{(\mu_i)\in M_n}\sup_{(\nu_j)\in M_n}\sum_{i=1}^{b^n}\sum_{j=1}^{b^n}|\mu_i||\nu_j|\big|\rho^{(0,0)}_{i,j}\big|\le L\lambda^n.
\label{bbb}
\end{align}
\end{lemma}

\begin{proof}By considering the case $k=0$ in \Cref{l9} (a) and using the triangle inequality, it suffices to prove \Cref{bbb} for $\tilde \rho^{(0,0)}_{i,j}$ in place of $\rho^{(0,0)}_{i,j}$. 
Indeed, from \eqref{l9 eq1} and \eqref{Mn Hoelder eq},
\begin{align*}
\sup_{(\mu_i)\in M_n}\sup_{(\nu_j)\in M_n}\sum_{i=1}^{b^n}\sum_{j=1}^{b^n}|\mu_i||\nu_j|\big|\rho^{(0,0)}_{i,j}-\tilde \rho^{(0,0)}_{i,j}\big|\le L b^{2nH}(b^{-2n\tau}+b^{-n(\tau+H)})=Lb^{-n(\tau-H)}.
\end{align*}
Note also that $\tilde \rho^{(0,0)}_{i,j}$ only involves standard increments of $W$, i.e., 
$$\tilde \rho^{(0,0)}_{i,j}=\bE_W[(W(t_i)-W(t_{i-1}))(W(t_j)-W(t_{j-1}))].
$$

Consider first the case $H{\ge}1/2$. We bound each factor $\mu_i,\nu_j$ by $\pm 1$ and use Cauchy--Schwarz to obtain bounds on the near-diagonal terms:
\begin{align*}
\sup_{(\mu_i)\in M_n}\sup_{(\nu_j)\in M_n}\sum_{\substack{1\leq i,j\leq b^n\\ |i-j|\leq 2}}|\mu_i||\nu_j|\big|\rho^{(0,0)}_{i,j}\big|&\leq \sum_{\substack{1\leq i,j\leq b^n\\ |i-j|\leq 2}}|\tilde \rho^{(0,0)}_{i,j}|\\
&\leq Lb^{n}\Big(\bE_W[(W(t_i)-W(t_{i-1}))^2]\cdot\bE_W[(W(t_j)-W(t_{j-1}))^2]\Big)^{1/2}\\
&\leq Lb^{n}b^{-2nH}.
\end{align*}
By repeated use of H\"{o}lder's inequality, we estimate the remaining terms as follows,
\begin{align*}
S_n&:=\sup_{(\mu_i)\in M_n}\sup_{(\nu_j)\in M_n}\sum_{i=1}^{b^n}|\mu_i|\sum_{\substack{1\leq j\leq b^n\\ |i-j|> 2}}|\nu_j||\tilde \rho^{(0,0)}_{i,j}|\\
&\leq \sup_{(\mu_i)\in M_n}\sup_{(\nu_j)\in M_n}\bigg(\sum_{i=1}^{b^n}|\mu_i|^q\bigg)^{1/q}\bigg(\sum_{i=1}^{b^n}\bigg(\sum_{\substack{1\leq j\leq b^n\\ |i-j|> 2}}|\nu_j||\tilde \rho^{(0,0)}_{i,j}|\bigg)^p\bigg)^{1/p}\\
&\leq \sup_{(\mu_i)\in M_n}\sup_{(\nu_j)\in M_n}\bigg(\sum_{i=1}^{b^n}\bigg(\bigg(\sum_{\substack{1\leq j\leq b^n\\ |i-j|> 2}}|\nu_j|^q\bigg)^{1/q}\bigg(\sum_{\substack{1\leq j\leq b^n\\ |i-j|> 2}}|\tilde \rho^{(0,0)}_{i,j}|^p\bigg)^{1/p}\bigg)^p\bigg)^{1/p}\\
&\leq \bigg(\sum_{i=1}^{b^n}\sum_{\substack{1\leq j\leq b^n\\ |i-j|> 2}}|\tilde \rho^{(0,0)}_{i,j}|^p\bigg)^{1/p}.
\end{align*}
For each fixed $1\leq i\leq b^n$, we apply \Cref{cov} with $m=k=n$ to obtain
$$
\sum_{\substack{1\leq j\leq b^n\\ |i-j|> 2}}|\tilde \rho^{(0,0)}_{i,j}|^p\leq L\sum_{\ell=2}^{b^n}\bigg(b^{-2n}(b^{-n}\ell)^{(2H-2)}\bigg)^p.
$$
Summation over $i$ and recalling that $p=1/H$ yields
\begin{align*}
S_n\leq L\bigg(\sum_{i=1}^{b^n}\sum_{\ell=2}^{b^n}\bigg(b^{-2n}(b^{-n}\ell)^{(2H-2)}\bigg)^p\bigg)^{1/p}\leq L\bigg(b^{-n}\sum_{\ell=2}^{b^n}\ell^{2-2p}\bigg)^H\leq Lb^{{2n(H-1)}}.\end{align*}

Now we consider the case $0<H{<} 1/2$. Then $1<q\leq 2$ and $M_n$ is contained in the unit ball, $B_1$, of $\bR^{b^n}$. Using that $|\tilde \rho^{(0,0)}_{i,j}|\le b^{-2nH}$ by the Cauchy--Schwarz inequality, the near-diagonal terms can be bounded as follows,
\begin{align*}
\sup_{(\mu_i)\in M_n}\sup_{(\nu_j)\in M_n}\sum_{\substack{1\leq i,j\leq b^n\\ |i-j|\leq 2}}|\mu_i||\nu_j|\big|\tilde\rho^{(0,0)}_{i,j}\big|
&\leq\sup_{(\mu_i)\in M_n}\sup_{(\nu_j)\in M_n}\sum_{i=1}^{b^n}\sum_{\substack{1\leq j\leq b^n\\ |i-j|\leq 2}}(|\mu_i|^2+|\nu_j|^2)|\tilde \rho^{(0,0)}_{i,j}|\\
&\leq b^{-2nH}\sup_{\substack{(\mu_i)\in B_1\\ (\nu_j)\in B_1}}\sum_{i=1}^{b^n}\sum_{\substack{1\leq j\leq b^n\\ |i-j|\leq 2}}(|\mu_i|^2+|\nu_j|^2)\\
&\leq Lb^{-2nH}.
\end{align*}
To deal with the off-diagonal terms, we replace the  covariance matrix $\{|\tilde \rho^{(0,0)}_{i,j}|\}_{1\leq i,j\leq b^n}$ with
$$R:=\{|\tilde \rho^{(0,0)}_{i,j}|\bone_{|i-j|>2}\}_{1\leq i,j\leq b^n}.$$
By \Cref{l5} and Lemma 10.2.1 in~\cite{MarcusRosen}, we conclude that the operator norm of $R$ satisfies  $\Vert R\Vert\le Lb^{-2nH}$. Thus
\begin{align*}
\sup_{(\mu_i)\in M_n}\sup_{(\nu_j)\in M_n}\sum_{\substack{1\leq i,j\leq b^n\\ |i-j|> 2}}|\mu_i||\nu_j|\big|\tilde\rho^{(0,0)}_{i,j}\big|&\leq\sup_{\substack{(\mu_i)\in B_1\\ (\nu_j)\in B_1}}\sum_{i=1}^{b^n}\sum_{j=1}^{b^n}|\mu_i||\nu_j|\big|\tilde\rho^{(0,0)}_{i,j}\big|\bone_{|i-j|>2}=\Vert R\Vert\le Lb^{-2nH},
\end{align*}
as required.
\end{proof}

Now we prepare for the induction step. We start with the following lemma, which gives a key reason for why it is convenient to work with $\tilde\rho^{(m,k)}_{i,j}$ instead of $\rho^{(m,k)}_{i,j}$.

\begin{lemma}
\label{w}
Let $H\in(0,1),\ n\in\bN$, and $0\leq u\leq v\leq 1$ be fixed where $ub^n,vb^n\in\bZ$, then
\begin{align*}
&\hspace{0.5cm}\sum_{j=1}^{b^{n+1}}|\mu_j||\bE_W[(W(jb^{-n-1})-W((j-1)b^{-n-1}))(W(v)-W(u))]|\\
&\leq \sum_{i=1}^{b^n}\bigg(\max_{(i-1)b< j\leq ib}|\mu_{j}|\bigg)|\bE_W[(W(ib^{-n})-W((i-1)b^{-n}))(W(v)-W(u))]|.
\end{align*}
\end{lemma}

\begin{proof}
For each fixed $i$, the intervals $((i-1)b^{-n},ib^{-n})$ and $(u,v)$ either have containment relationship or are disjoint. Hence, for subintervals $[s,t]\se [(i-1)b^{-n},ib^{-n}]$, the sign of the covariance of  $W(t)-W(s)$ and $W(v)-W(u)$ is independent of the choice of $s$ and $t$. Indeed, it is well known that $W(t)-W(s)$ and $W(v)-W(u)$ are always positively correlated if the intervals $((i-1)b^{-n},ib^{-n})$ and $(u,v)$ have containment relationship; if they are disjoint, then $W(t)-W(s)$ and $W(v)-W(u)$ are positively correlated if and only if $H>1/2$, negatively correlated if and only if $H<1/2$, and independent if $H=1/2$.  Thus the claim follows by removing the absolute values and using a telescopic sum.
\end{proof}

\begin{lemma}
Let $L>0$ and $\lambda\in(0,1)$ be constants and $(a_n)_{n\in\bN}$ be a sequence of positive real numbers satisfying $a_0\leq L$ and 
$$
a_{n+1}\leq L\lambda^n+\lambda a_n,\qquad n\in\bN_0.
$$
Then there are constants $L_1>0$ and $\lambda_1\in(0,1)$ such that $a_n\leq L_1\lambda_1^n$.\label{p}
\end{lemma}

\begin{proof}
Dividing both sides by $\lambda^{n+1}$ we see that $b_n:=a_n/\lambda^n$ satisfies $b_{n+1}\leq b_n+L/\lambda$. This gives $b_n\leq L+nL/\lambda$ so that $a_n\leq Ln\lambda^n\leq L_1\lambda_1^n$. 
\end{proof}

The following is an induction argument using \Cref{ind1} as the base case. It states in particular that the contribution from the terms with $m=0$ in \eqref{s2} is of the order $O(\lambda^n)$.

\begin{lemma}\label{l7}
Let
\begin{equation}\label{F def eq}
F_{n,k}:=\sup_{(\nu_i)\in M_n}\sup_{(\mu_i)\in M_n}\sum_{i=1}^{b^n}\sum_{j=1}^{b^n}|\mu_i\nu_j||\rho^{(0,k)}_{i,j}|\quad\text{and}\quad  F_n:=\sum_{k=0}^{n-1}\alpha^k F_{n,k}.
\end{equation}
Then there exists $\lambda\in(0,1)$ such that $F_n=O(\lambda^n)$ 
for all $H\in(0,1)$.
\end{lemma}

\begin{proof} Let us define $\wt F_{n,k}$ and $\wt F_n$ as in \eqref{F def eq}, but with $\rho^{(0,k)}_{i,j}$ replaced with $\tilde \rho^{(0,k)}_{i,j}$.
By \Cref{l9} (b) and the triangle inequality, the assertion will follow if we can show that $\wt F_n=O(\lambda^n)$. 
To this end, for $s\ge0$, we write $M_{n,s}:=\{(s^{1/q}\mu_i):(\mu_i)\in M_n\}$. Let us define for $r,s\geq 0$ the function
\begin{equation*}
\wt F_{n,k}(r,s):=\sup_{(\nu_i)\in M_{n,r}}\sup_{(\mu_i)\in M_{n,s}}\sum_{i=1}^{b^n}\sum_{j=1}^{b^n}|\mu_i\nu_j||\tilde\rho^{(0,k)}_{i,j}|.
\end{equation*}
Obviously we have $\wt F_{n,k}(1,1)=\wt F_{n,k}$ as well as the homogeneity property
\begin{align}
\wt F_{n,k}(r,s)=(rs)^{1/q}\wt F_{n,k}(1,1).\label{ho}
\end{align}
 For the induction step we will bound $\wt F_{n+1,k+1}$ from above by $\wt F_{n,k}$:
\begin{align*}
\lefteqn{\wt F_{n+1,k+1}}\\
&=\sup_{(\mu_i)\in M_{n+1}}\sup_{(\nu_j)\in M_{n+1}}\sum_{i=1}^{b^{n+1}}\sum_{j=1}^{b^{n+1}}|\mu_i\nu_j|\big|\bE[(W(t^{(n+1)}_i)-W(t^{(n+1)}_{i-1}))\Delta W(b^{k+1}t^{(n+1)}_{j-1},b^{k+1}t^{(n+1)}_j)]\big|\\
&=\sup_{(\mu_i)\in M_{n+1}}\sup_{(\nu_j)\in M_{n+1}}\sum_{i=1}^{b^{n+1}}\sum_{v=0}^{b-1}\sum_{j=1}^{b^{n}}|\mu_i||\nu_{j+vb^n}|\big|\bE[(W(t^{(n+1)}_i)-W(t^{(n+1)}_{i-1}))\Delta W(b^kt^{(n)}_{j-1},b^kt^{(n)}_j)]\big|\\
&=\sup_{(\mu_i)\in M_{n+1}}\sup_{(\nu_j)\in M_{n+1}}\sum_{v=0}^{b-1}\sum_{j=1}^{b^{n}}|\nu_{j+vb^n}|\sum_{i=1}^{b^{n+1}}|\mu_i|\big|\bE[(W(t^{(n+1)}_i)-W(t^{(n+1)}_{i-1}))\Delta W(b^kt^{(n)}_{j-1},b^kt^{(n)}_j)]\big|\\
&\leq\sup_{(\mu_i)\in M_{n+1}}\sup_{(\nu_j)\in M_{n+1}}\sum_{v=0}^{b-1}\sum_{j=1}^{b^{n}}|\nu_{j+vb^n}|\sum_{i=1}^{b^{n}}\max_{(i-1)b<\ell\le ib}|\mu_{\ell}|\big|\bE[(W(t^{(n)}_i)-W(t^{(n)}_{i-1}))\Delta W(b^kt^{(n)}_{j-1},b^kt^{(n)}_j)]\big|\\
&=\sup_{(\mu_i)\in M_{n+1}}\sup_{(\nu_j)\in M_{n+1}}\sum_{v=0}^{b-1}\sum_{j=1}^{b^{n}}|\nu_{j+vb^n}|\sum_{i=1}^{b^{n}}|\tilde\mu_{i}|\big|\bE[(W(t^{(n)}_{i})-W(t^{(n)}_{i-1}))\Delta W(b^kt^{(n)}_{j-1},b^kt^{(n)}_j)]\big|,
\end{align*}
where the second step follows from the simple relation $\fp{(j+vb^n)b^{(k+1)-(n+1)}}=\fp{jb^{k-n}}$, the fourth step follows from \Cref{w}, and where we define $\tilde\mu_{i}:=\max_{(i-1)b<\ell\le ib}|\mu_{\ell}|$.
Then $\sum_{i=1}^{b^{n+1}}|\mu_{i}|^q\leq 1$ implies $\sum_{i=1}^{b^n}|\tilde\mu_{i}|^q\leq 1$ so that $(\tilde\mu_i)\in M_n$. For ${\bf s}:=(s_0,\dots,s_{b-1})$ in the $b$-dimensional simplex $\bS_b=\{(s_0,\dots,s_{b-1}):s_i\ge0,\ \sum_is_i\le1\}$, we define
\begin{align}
I({\bf s}):=\bigg\{(\nu_j)\in M_{n+1}: \sum_{j=1}^{b^n}|\nu_{j+vb^n}|^q=s_v\text{ for  $v=0,\dots,b-1$}\bigg\}.\label{isit}
\end{align}
Then
\begin{align*}
\lefteqn{\wt F_{n+1,k+1}}\\
&\leq \sup_{{\bf s}\in\bS_b}\sum_{v=0}^{b-1}\sup_{(\nu_j)\in I({\bf s})}\sup_{(\mu_i)\in M_{n+1}}\sum_{j=1}^{b^{n}}|\nu_{j+vb^n}|\sum_{i=1}^{b^{n}}|\tilde\mu_i|\big|\bE[(W(t^{(n)}_{i})-W(t^{(n)}_{i-1}))\Delta W(b^kt^{(n)}_{j-1},b^kt^{(n)}_j)]\big|\\
&\leq \sup_{{\bf s}\in\bS_b}\sum_{v=0}^{b-1}\sup_{(\nu_j)\in I({\bf s})}\sup_{(\hat\mu_i)\in M_n}\sum_{j=1}^{b^{n}}|\nu_{j+vb^n}|\sum_{i=1}^{b^{n}}|\hat\mu_i|\big|\bE[(W(t^{(n)}_{i})-W(t^{(n)}_{i-1}))\Delta W(b^kt^{(n)}_{j-1},b^kt^{(n)}_j)]\big|\\
&= \sup_{{\bf s}\in\bS_b}\sum_{v=0}^{b-1}\wt F_{n,k}(1,s_v)= \wt F_{n,k}(1,1)\sup_{{\bf s}\in\bS_b}\sum_{v=0}^{b-1}s_v^{1/q}=b^H\wt F_{n,k},
\end{align*}
where the last line follows from \eqref{ho} and H\"{o}lder's inequality:
\begin{align}
\sum_{v=0}^{b-1}s_v^{1/q}\leq \bigg(\sum_{v=0}^{b-1}s_v\bigg)^{1/q}b^{1/p}\leq b^H.\label{hol}
\end{align}Hence,
\begin{align}
\wt F_{n+1}=\sum_{k=0}^n\alpha^k \wt F_{n+1,k}(1,1)&\leq \wt F_{n+1,0}(1,1)+\alpha\sum_{k=0}^{n-1}\alpha^k \wt F_{n+1,k+1}(1,1)\nonumber\\
&\leq \wt F_{n+1,0}(1,1)+\alpha b^H \sum_{k=0}^{n-1}\alpha^k \wt F_{n,k}(1,1)= \wt F_{n+1,0}(1,1)+\alpha b^H \wt F_n.\label{Fn}
\end{align}
By \Cref{ind1}, $\wt F_{n+1,0}(1,1)=O(\lambda^n)$, thus we conclude by \Cref{p} and $\alpha b^H<1$ that $\wt F_{n+1}=O(\lambda^n)$, completing the proof.
\end{proof}

Now we turn to the second induction step. This time we use \Cref{l7} as the base case:

\begin{proposition}Let $H\in(0,1)$, then $\sigma_n^2=O(\lambda^n)$ for some $\lambda\in(0,1)$.
\label{p4}
\end{proposition}

\begin{proof} Similarly to the proof of \Cref{l7}, we define  the function
$$
G_{n,m,k}(r,s):=\sup_{(\mu_i) \in M_{n,r}}\sup_{(\nu_j)\in M_{n,s}}\sum_{i=1}^{b^n}\sum_{j=1}^{b^n}\mu_i\nu_j\rho^{(m,k)}_{i,j}.
$$
From  \eqref{s2}, we get
\begin{align*}
\sigma_n^2\leq \sum_{m=0}^{n-1}\sum_{k=0}^{n-1}\alpha^{m+k}G_{n,m,k}(1,1)=:G_n.
\end{align*}
So it suffices to show $G_n=O(\lambda^n)$. As in the preceding proof, $G_{n,m,k}$ satisfies the following  homogeneity property,
\begin{align}
G_{n,m,k}(r,s)=(rs)^{1/q}G_{n,m,k}(1,1).\label{hom}
\end{align}
Let us also introduce the shorthand notation
$$\Delta_2B(s,t,u,v):=\Delta B(s,t)\Delta B(u,v).
$$
Then, with $\bS_b$ denoting again the $b$-dimensional standard simplex and $I({\bf s})$ as in \eqref{isit},
\begin{align*}
\lefteqn{G_{n+1,m+1,k+1}(1,1)}\\&=\sup_{(\mu_i)\in M_{n+1}}\sup_{(\nu_j)\in M_{n+1}}\sum_{i=1}^{b^{n+1}}\sum_{j=1}^{b^{n+1}}\mu_i\nu_j\bE[\Delta_2 B(b^{m+1}t^{(n+1)}_{i-1},b^{m+1}t^{(n+1)}_i,b^{k+1}t^{(n+1)}_{j-1},b^{k+1}t^{(n+1)}_j)]\\
&=\sup_{(\mu_i)\in M_{n+1}}\sup_{(\nu_j)\in M_{n+1}}\sum_{u=0}^{b-1}\sum_{v=0}^{b-1}\sum_{i=1}^{b^{n}}\sum_{j=1}^{b^{n}}\mu_{i+ub^n}\nu_{j+vb^n}\bE[\Delta_2 B(b^{m+1}t^{(n+1)}_{i+ub^n-1},b^{m+1}t^{(n+1)}_{i+ub^n},b^{k+1}t^{(n+1)}_{j+vb^n-1},b^{k+1}t^{(n+1)}_{j+vb^n})]\\
&=\sup_{(\mu_i)\in M_{n+1}}\sup_{(\nu_j)\in M_{n+1}}\sum_{u=0}^{b-1}\sum_{v=0}^{b-1}\sum_{i=1}^{b^{n}}\sum_{j=1}^{b^{n}}\mu_{i+ub^n}\nu_{j+vb^n}\bE[\Delta_2B(b^m t^{(n)}_{i-1},b^m t^{(n)}_i,b^k t^{(n)}_{j-1},b^k t^{(n)}_j)]\\
&\leq\sup_{{\bf r}\in\bS_b}\sup_{{\bf s}\in\bS_b}\sum_{u=0}^{b-1}\sum_{v=0}^{b-1}\sup_{(\mu_i)\in I({\bf r})}\sup_{(\nu_j)\in I({\bf s})}\sum_{i=1}^{b^{n}}\sum_{j=1}^{b^{n}}\mu_{i+ub^n}\nu_{j+vb^n}\rho^{(m,k)}_{i,j}.
\end{align*}
 If $u,v\in[0,b-1]\cap\bZ$ are given and $(\mu_i)\in I({\bf r})$ and $(\nu_j)\in I({\bf s})$, then by definition
\begin{align*}\sum_{i=1}^{b^{n}}\sum_{j=1}^{b^{n}}\mu_{i+ub^n}\nu_{j+vb^n}\rho^{(m,k)}_{i,j}\leq G_{n,m,k}(r_u,s_v).
\end{align*}
Thus, by the homogeneity property \eqref{hom},
\begin{align*}
G_{n+1,m+1,k+1}(1,1)&\leq \sup_{{\bf r}\in \bS_b}\sup_{{\bf s}\in\bS_b}\sum_{u=0}^{b-1}\sum_{v=0}^{b-1}G_{n,m,k}(r_u,s_v)\\
&=\sup_{{\bf r}\in\bS_b}\sup_{{\bf s}\in\bS_b}\sum_{u=0}^{b-1}\sum_{v=0}^{b-1}(r_us_v)^{1/q}G_{n,m,k}(1,1)\\
&=G_{n,m,k}(1,1)\bigg(\sup_{{\bf s}\in\bS_b}\sum_{u=0}^{b-1}r_u^{1/q}\bigg)\bigg(\sup_{{\bf s}\in\bS_b}\sum_{v=0}^{b-1}s_v^{1/q}\bigg)\\
&= b^{2H}G_{n,m,k}(1,1),
\end{align*}where the final step follows from \eqref{hol}. Next, observe that $G_{n,m,k}=G_{n,k,m}$ so that
\begin{align}
G_{n+1}&=\sum_{m=0}^{n}\sum_{k=0}^{n}\alpha^{m+k}G_{n+1,m,k}(1,1)\nonumber\\
&=\sum_{\substack{0\leq m,k\leq n\\ mk=0}}\alpha^{m+k}G_{n+1,m,k}(1,1)+\alpha^2\sum_{m=0}^{n-1}\sum_{k=0}^{n-1}\alpha^{m+k}G_{n+1,m+1,k+1}(1,1)\nonumber\\
&\leq L\sum_{k=0}^{n}\alpha^{k}G_{n+1,0,k}(1,1)+\alpha^2b^{2H}\sum_{m=0}^{n-1}\sum_{k=0}^{n-1}\alpha^{m+k}G_{n,m,k}(1,1)\nonumber\\
&= L\sum_{k=0}^{n}\alpha^{k}G_{n+1,0,k}(1,1)+\alpha^2b^{2H}G_n.\label{gn}
\end{align}
By \Cref{l7}, 
$$\sum_{k=0}^{n}\alpha^{k}G_{n+1,0,k}(1,1)\leq \sum_{k=0}^{n}\alpha^{k}F_{n+1,k}=F_{n+1}\le L\lambda^n.$$ 
Since $\alpha^2b^{2H}<1$,  \Cref{p} now yields $G_n\le L\lambda^n$.
\end{proof}

Combining \Cref{p4} and \Cref{l6} proves  \Cref{main thm} (b) in the case $t=1$. In the next subsection, we sketch a proof for the case $0\le t<1$.

\subsubsection{Linearity of the $p$-th variation}\label{linear section}

Now we sketch how the preceding arguments can be modified so as to obtain a proof of \Cref{main thm} (b) for all $t\in[0,1]$. The details will be left to the reader. 
We consider $M\in\bN$, $r \in\bN_0$, and an interval $I=[r  b^{-M},(r +1)b^{-M}]\se [0,1]$. The goal is to show that the $p$-th variation of $X$  on $I$ is equal to $\<X\>^{(p)}_1$ times the length of $I$.
For given $n\in\bN$, 
the $n^{\text{th}}$ order approximation of the $p$-th variation of $X$  on $I$ is then
\begin{align*}
V_{I,n}&:=\sum_{k=0}^{b^{n-M}-1}|X(r  b^{-M}+(k+1)b^{-n})-X(r  b^{-M}+kb^{-n})|^p\\
&=\sum_{k=0}^{b^{n-M}-1}\bigg|\sum_{m=0}^{n-1}\alpha^m (B(\fp{r  b^{m-M}+(k+1)b^{m-n}})-B(\fp{r  b^{m-M}+kb^{m-n}}))\bigg|^p\\
&=\sum_{k=0}^{b^{n-M}-1}\bigg|\sum_{m=1}^{n}\alpha^{n-m} (B(\fp{(r  b^{n-M}+(k+1))b^{-m}})-B(\fp{(r  b^{n-M}+k)b^{-m}})\bigg|^p\\
&=b^{-M}(\alpha^p b)^n\bE_R\bigg[\bigg|\sum_{m=1}^n\alpha^{-m}(B(\fp{(S_n+1)b^{-m}})-B(\fp{S_nb^{-m}})\bigg|^p\bigg],
\end{align*}
where $S_n$ is a random variable on $(\Omega_R,\cF_R,\bP_R)$ with a uniform  distribution on $\{rb^{n-M},rb^{n-M}+1,\cdots,(r+1)b^{n-M}-1\}$. Note that our random variables $(R_m)$ were constructed in such a way that $R_mb^{-m}=\fp{R_nb^{-m}}$. So all we need is to replace in Sections \ref{Convergence in expectation section} and 
\ref{conce}  all terms of the form $R_mb^{-m}$ with $\fp{S_nb^{-m}}$ and verify that all arguments still go through. Indeed, the expectation $\bE_W[V_{I,n}]$ can  be analyzed exactly as in \Cref{p1} and \Cref{c1}, and one obtains
$$\bE_W[V_{I,n}]=b^{-M}\frac{(1+o(1))c_H}{(1-\alpha^2b^{2H})^{p/2}}.
$$
Note that the factor $b^{-M}$ is just the length of $I$.  
For the concentration inequality, we simply restrict the sequence $(\mu_k)$ to the indices $\{k:[t_{k-1},t_k]\se I\}$.

\subsection{Proof of \Cref{critical case H=1/2 thm} }\label{proof of critical case H=1/2 thm section}

For simplicity, we only consider the case $t=1$. The extension to the case $0<t<1$ can be obtained in the same way as at the end of \Cref{thm (c) proof section}. 

Next, we claim that we may assume without loss of generality that $\kappa$ is equal to the standard choice \eqref{standard kappa}, which for $H=1/2$ is simply given by $\kappa(t)=t$. To this end, let $B(t)=W(t)-\kappa(t)W(1)$ be the Brownian bridge with a generic function $\kappa:[0,1]\to[0,1]$ satisfying $\kappa(0)=0$ and $\kappa(1)=1$ and being H\"older continuous with exponent $\tau>1/2$. The corresponding Wiener--Weierstra\ss\ bridge is denoted by $X$. For the moment, we denote by $\wt B(t)=W(t)-tW(1)$ the standard Brownian bridge and let  $\wt X$  be the corresponding processes. Then $\wt X=X+f\cdot W(1)$, where $f(t)=\sum_{n=0}^\infty \alpha^n\phi(\fp{b^n t})$ for 
$\phi(t):=t-\kappa(t)$. Since $\phi$ is H\"older continuous with exponent $\tau>1/2$ and $\alpha^2b=1$, \Cref{p variation prop} yields that $f$ has a finite quadratic variation $\<f\>^{(2)}_1$. Thus, the following lemma yields that \eqref{critical case H=1/2 eqn} 
 holds for $X$ if and only if it holds for $\wt X$. The assertion for $p> 2$ is obtained in a similar way from Lemma 2.4 in \cite{SchiedZZhang}. Thus, we may assume in the sequel that $\kappa(t)=t$. 

\begin{lemma}\label{qv lemma} For any function $h\in C[0,1]$ and $n\in\bN$, let us denote 
$$V_n(h):=\sum_{k=0}^{b^n-1}\big(h((k+1)b^{-n})-h(kb^{-n})\big)^2.
$$
Now suppose that $f,g\in C[0,1]$ are functions for which $\limsup_n V_n(f)<\infty$ and $n^{-1}V_n(g)\to c$, where $c\ge0$.  Then $n^{-1}V_n(f+g)\to c$.
\end{lemma}

\begin{proof} We have 
$$\frac1nV_n(f+g)=\frac1n V_n(f)+\frac1nV_n(g)+\frac2n \sum_{k=0}^{b^n-1}\big(f((k+1)b^{-n})-f(kb^{-n})\big)\big(g((k+1)b^{-n})-g(kb^{-n})\big).
$$
The Cauchy--Schwarz inequality implies that the absolute value of the rightmost term is bounded by $2n^{-1}V_n(f)^{1/2}V_n(g)^{1/2}$, and this expression converges to zero by our assumptions.
\end{proof}

To prove the assertion of \Cref{critical case H=1/2 thm}, we need to show \eqref{critical case H=1/2 eqn} and, in addition, that the $p$-th variation of $X$ vanishes for $p>2$; that the $p$-th variation of $X$ for $p<2$ is infinite will then follow from \eqref{critical case H=1/2 eqn}  by using the argument in  the final step in the proof of Theorem 2.1 in~\cite{MishuraSchied2}. As in the proof of \Cref{main thm}, we show first convergence of the  expectation.
\Cref{l1} states that the expected $p$-th variation  is of the form
$$
\bE_W[V_n]=(\alpha^pb)^n\bE\bigg[\bigg|\int_0^1f_n(x)\,dW(x)\bigg|^p\bigg],
$$
where $\bE[\,\cdot\,]$ denotes the expectation with respect to $\bP=\bP_W\otimes\bP_R$ and $f_n(x)=g_n(x)-h_n$ for
 $$g_n(x):=\sum_{m=1}^n\alpha^{-m}\bone_{[R_mb^{-m},(R_m+1)b^{-m}]}(x)\quad\text{and}\quad
h_n:=\sum_{m=1}^n\alpha^{-m}((R_m+1)b^{-m}-R_mb^{-m}).$$
In our present case, we have $\alpha^2b=1=\alpha b^{1/2}$, and so the factor $(\alpha^pb)^n$ in front of the expectation is equal to 1 for $p=2$ and it decreases geometrically for $p>2$.  
In the proof of \Cref{p1},  Equations \eqref{expectation factor eq} and 
\eqref{d0}  are still valid. However, 
  the diagonal terms \eqref{d1} are  simply equal to $n$ and so $C_n$ from \eqref{d1} must now be replaced with $C_n=n$. \Cref{d1a} thus becomes
  $$
\bE\bigg[\bigg|\int_0^1g_n(x)\,dW(x)\bigg|^p\bigg]=c_H\bE_R\big[\big|n+\eta_n\big|^{p/2}\big].
$$
Note next that $c_H=1$ in our case $H=1/2$. 
Moreover, \Cref{etan expectation eq} remains valid, but only the case $\alpha b>1=\alpha^2b$ can occur, so that $\bE[|\eta_n|]=O(1)$. 
   Using \eqref{d1a}, one thus shows by using similar arguments as in \Cref{l2} that for $p=2$,
  \begin{equation}\label{crit case gn conv eq}
  \frac1n\bE\bigg[\bigg|\int_0^1g_n(x)\,dW(x)\bigg|^2\bigg]=\frac{1}{n}\bE_R\big[\big|n+\eta_n\big|\big]\longrightarrow 1.
    \end{equation}
For $p>2$, on the other hand, there exists $\lambda\in(0,1)$ such that
$$(\alpha^pb)^n\bE\bigg[\bigg|\int_0^1g_n(x)\,dW(x)\bigg|^p\bigg]=o(\lambda^n).
$$
Finally, one shows just as in \Cref{c1} that $g_n$ can be replaced with $f_n$ in \eqref{crit case gn conv eq}. Altogether, this yields that $\frac1n\bE[V_n]\to1$ for $p=2$ and $\bE[V_n]=o(\lambda^n)$ for $p>2$.

Having established the convergence of the  expectation, we now turn toward the almost sure convergence. For $p>2$, we have $\bE[V_n ]\le L\lambda^n$ for some $L>0$ and $\lambda\in(0,1)$. We choose $\nu\in(\lambda,1)$ and apply Markov's inequality to get
$$\bP(V_n\ge\nu^n)\le \frac{\bE[V_n]}{\nu^n}\le L\Big(\frac\lambda\nu\Big)^n.
$$
Hence, the Borel-Cantelli lemma yields $V_n\to0$ $\bP$-a.s.

For $p=2$ we use again a concentration bound. However, the method used in the proof of  \Cref{main thm} does not work in the critical case. The main reason is that the inequality $\alpha b^H<1$ no longer holds, so that we are not able to conclude from \eqref{Fn} and \eqref{gn} that $\wt F_n$ and $G_n$ decay geometrically. We therefore use a somewhat different approach here. First, we fix $n$ and let again $t_i:=ib^{-n}$. Following the proof of Theorem 1 in~\cite{KleinGine}, we fix $n$ and let $(Y_i)_{i=0,\dots, b^n-1}$ be an orthonormal basis for the linear hull of $\{X(t_1)-X(t_0),\dots,X(t_{b^n})-X(t_{b^n-1})\}$ in $L^2(\Omega_W,\cF_W,\bP_W)$. Then we let $B=(b_{i,j})_{i,j=1,\dots, b^n}$ denote the matrix with entries $b_{i,j}=\bE_W[(X(t_i)-X(t_{i-1}))Y_j]$. Then we define $A=(a_{i,j})_{i,j=1,\dots, b^n}$ for $a_{i,j}=\bE_W[(X(t_i)-X(t_{i-1}))(X(t_j)-X(t_{j-1}))]$ and observe that $A=BB^\top$. Finally, we define $C:=B^\top B$. Since $V_n=Y^\top CY$, Hanson and Wright's bound \cite{HansonWright} yields that there are constants $L_1,L_2>0$ such that for $s>0$,
\begin{equation*}
\bP_W\big(|V_n-\bE_W[V_n]|\ge s\big)\le2\exp\bigg({-\min\bigg\{\frac{L_1 s}{\|C\|},\frac{L_2 s^2}{\text{trace} \,C^2}\bigg\}}\bigg),
\end{equation*}
where $\|C\|$ is the spectral radius of $C$. Then one argues as in the proof of Theorem 1 in~\cite{KleinGine} that 
\begin{equation*}
\text{trace} \,C^2\le  \|C\|\cdot\text{trace} \,A=  \|C\|\cdot {\bE_W[V_n]}\qquad\text{and}\qquad  \|C\|\le \inf_{m\in\bN}(\text{trace}( A^m))^{1/m}.
\end{equation*}
Since we have seen above that $\frac1n\bE_W[V_n]\to 1$, there is a constant $L_3$ such that  
\begin{equation*}
\bP_W\big(|V_n-\bE_W[V_n]|\ge s\big)\le2\exp\bigg(-\frac{L_3({s\wedge \frac{s^2}n})}{\inf_{m\in\bN}(\text{trace}( A^m))^{1/m}}\bigg).
\end{equation*}
We will show below that there is a constant $L>0$ such that 
for sufficiently large $n$,
\begin{equation}\label{A traces eq}
\inf_{m\in\bN}(\text{trace}( A^m))^{1/m}\le L.\end{equation}
Hence, for those $n$,
$$\bP_W\Big(\Big|\frac1nV_n-\frac1n\bE_W[V_n]\Big|\ge n^{-1/4}\Big)\le 2\exp\Big(-\frac{L_3}{L}\big(n^{3/4}\wedge n^{1/2}\big)\Big),
$$
and so the Borel--Cantelli lemma yields that $\lim_n\frac1n V_n=\lim_n\frac1n \bE_W[V_n]=1$.

It remains to establish \eqref{A traces eq}. 
We first  need upper bounds for $|a_{i,j}^{(n)}|$. Recalling the shorthand notation \eqref{rhomkij}, we have 
\begin{align*}
\sum_{j=1}^{b^n}|a_{i,j} |
&=\sum_{j=1}^{b^n}\bigg|\sum_{m=0}^{n-1}\sum_{k=0}^{n-1}\alpha^{m+k}\rho_{i,j}^{(m,k)}\bigg|
\leq 2\sum_{0\leq k\leq m\leq n}\alpha^{m+k}\sum_{j=1}^{b^n}\big|\rho_{i,j}^{(m,k)}\big|.
\end{align*}
One checks that the intervals $(\fp{ t_{i-1} b^{m}} ,\fp{ t_ib^{m}} )$ and $(\fp{ t_{j-1}b^{k}} ,\fp{t_jb^{k}} )$ are either disjoint or have containment relationship. Using our assumption $\kappa(t)=t$, we find that in the disjoint case, 
$$
\big|\rho_{i,j}^{(m,k)}\big|=\left|\bE\left[\big(B(\fp{t_ib^{m}})-B(\fp{ t_{i-1} b^{m}})\big)\big(B(\fp{t_jb^{k}})-B(\fp{t_{j-1} b^{k}})\big)\right]\right|\leq b^{k+m-2n}.
$$
In the containment case, we have
$$
\big|\rho_{i,j}^{(m,k)}\big|=\left|\bE\left[\big(B(\fp{t_ib^{m}})-B(\fp{ t_{i-1} b^{m}})\big)\big(B(\fp{t_jb^{k}})-B(\fp{t_{j-1} b^{k}})\big)\right]\right|\leq {2b^{m\wedge k-n}}.
$$
When fixing $i,k,m$ and letting $j$ vary in $\{0,1,\dots,b^n-1\}$, $b^m$ choices of $j$ will give containment and $b^n-b^m$ others give disjointness. Thus, since $\alpha^2 b=1$,
\begin{align*}
\sum_{j=1}^{b^n}|a_{i,j} |&\leq 2\sum_{0\leq k\leq m\leq n-1}\alpha^{m+k}\big(b^m\cdot 2b^{k-n}+(b^n-b^m)\cdot b^{k+m-2n}\big)\\
&\leq 6b^{-n}\sum_{0\leq k\leq m\leq n-1}(\alpha b)^{m+k}\leq 6b^{-n}\Big(\sum_{m=0}^{n-1}(\alpha b)^m\Big)^2\le L_4\end{align*}
for some constant $L_4$.

Set 
$$
\Lambda_\ell :=\sum_{k_1=0}^{b^n-1}\cdots \sum_{k_{\ell +1}=0}^{b^n-1}|a_{k_1,k_2}|\cdots|a_{k_\ell ,k_{\ell +1}}|.
$$
We then have that\begin{align*}
\text{trace}( A^m)&\leq \Lambda_m=\sum_{k_1=0}^{b^n-1}\cdots \sum_{k_{m}=0}^{b^n-1}|a_{k_1,k_2}|\cdots|a_{k_{m-1},k_{m}}|\sum_{k_{m+1}=0}^{b^n-1}|a_{k_m,k_{m+1}}|\le \Lambda_{m-1}L_4.
\end{align*}
By induction,
\begin{align*}
\text{trace}( A^m)&\leq L_4^{m-1}\Lambda_1=L_4^{m-1}\sum_{k_1=0}^{b^n-1}\sum_{k_2=0}^{b^n-1}|a_{k_1,k_2}|\leq L_4^{m}b^n.
\end{align*}
We conclude that 
\begin{align*}
\inf_{m\in\bN}(\text{trace}( A^m))^{1/m} \leq \inf_{m\in\bN}L_4 b^{n/m}= L_4.
\end{align*}
This proves \eqref{A traces eq}.\qed

\subsection{Proof of \Cref{semimartingale thm}}

We note first that, if $X$ were a semimartingale, its sample paths would admit a continuous and finite quadratic variation. By \Cref{main thm}, we thus need only consider the case $H\wedge K=1/2$.  In the sequel, we are going to distinguish the cases $H=K$, $K<H$, and $H<K$. 

\begin{proof}[Proof of \Cref{semimartingale thm} for $H=K$] The assertion follows immediately from \eqref{critical case H=1/2 eqn}, because a continuous semimartingale cannot have infinite quadratic variation. 
\end{proof}

Now we turn to the case $K<H$. It needs the following preparation. Let $\bS$ be a partition of $[0,1]$. That is, there is $n\in\bN$ and $0=s_0<s_1<\cdots<s_n=1$ such that $\bS=\{s_0,\dots,s_n\}$. By $|\bS|=\max_i|s_i-s_{i-1}|$ we denote the mesh of $\bS$. For  functions $f\in C[0,1]$ and $\Psi:[0,\infty)\to [0,\infty)$, we define
$$v_\Psi(f,\bS):=\sum_{i=1}^n\Psi\big(|f(s_i)-f(s_{i-1})|\big)
$$
and
$$v_\Psi(f):=\lim_{\delta_\da0}\bigg(\sup\Big\{v_\Psi(f,\bS):\text{$\bS$ partition with $|\bS|\le\delta$}\Big\}\bigg).
$$
As a matter of fact, it is clearly sufficient if $\Psi$ is only defined on an interval $[0,x_0)$, provided that the mesh $|\bS|$ is sufficiently small.

\begin{lemma}\label{Phi var lemma} Consider the function 
 $$\Psi(0):=0\qquad\text{and}\qquad\Psi(x):=\frac{x^2}{2\log\log \frac1x},\qquad 0<x<1/e,
$$
and let $f,g\in C[0,1]$.
\begin{enumerate}
\item If $f$ is of bounded variation, then $v_\Psi(f)=0$.
\item If $f$ is H\"older continuous with exponent $1/2$, then  $v_\Psi(f)=0$.
\item If $v_\Psi(g)=0$, then $v_\Psi(f+g)=v_\Psi(f)$.
\end{enumerate}
\end{lemma}

\begin{proof} (a) Let $\bS=\{s_0,\dots,s_n\}$ be a partition with $|\bS|$ sufficiently small. Then
$$0\le v_\Psi(f,\bS)\le\bigg(\max_j\frac{|f(s_j)-f(s_{j-1})|}{2\log\log |f(s_j)-f(s_{j-1})|^{-1}}\bigg)\sum_{i=1}^n|f(s_i)-f(s_{i-1})|.
$$
As $|\bS|\to0$, the sum on the right converges to the total variation of $f$, and hence to a finite number. The maximum, on the other hand, tends to zero as $|\bS|\to0$ (here we use the conventions $1/0=\infty$, $\log0=-\infty$,  and $\log\infty=\infty$). 

(b) Let $c>0$ be such that $|f(t)-f(s)|\le c|t-s|^{1/2}$. Then, for $\bS=\{s_0,\dots,s_m\}$, 
\begin{align*}
0\le v_\Psi(f,\bS)\le\bigg(\max_j\frac{1}{2\log\log |f(s_j)-f(s_{j-1})|^{-1}}\bigg)\sum_{i=1}^nc|s_i-s_{i-1}|.
\end{align*}
Clearly, the value of the telescopic sum is $c$, while the maximum tends to zero as $|\bS|\to0$.

(c) One checks that $\Psi$ is increasing and strictly convex. Now we let $x_0:=1/(2e)$ and define $\Psi_0$ as that function which is equal to $\Psi$ on $[0,x_0]$ and, for $x>x_0$,  equal to $\Psi(x_0)(x/x_0)^a$ for $a=2+1/((1+\log2)\log(1+\log2))$. Then  there exists $\delta>0$ such that $v_{\Psi_0}(h,\bS)=v_\Psi(h,\bS)$ holds for all functions $h\in\{f,g,f+g\}$ and partitions $\bS$ with $|\bS|<\delta$. Next, one checks that 
 $\log\big(\Psi_\delta(e^x)\big)$ is strictly increasing and convex. Thus,  we may apply Mulholland's extension of Minkowski's inequality \cite{Mulholland}. In our context, it implies that 
\begin{equation}\label{Mulholland ineq}
\Psi_0^{-1}\big(v_{\Psi_0}(f,\bS)\big)-\Psi_0^{-1}\big(v_{\Psi_0}(g,\bS)\big)\le \Psi_0^{-1}\big(v_{\Psi_0}(f+g,\bS)\big)\le \Psi_0^{-1}\big(v_{\Psi_0}(f,\bS)\big)+\Psi_0^{-1}\big(v_{\Psi_0}(g,\bS)\big).
\end{equation}
Taking a sequence of partitions $(\bS_n)_{n\in\bN}$ with $|\bS_n|\to0$ and $v_{\Psi_0}(f+g,\bS_n)\to v_{\Psi_0}(f+g)$, applying the right-hand side of \eqref{Mulholland ineq}, and passing to the limit as $n\ua\infty$ thus yields that 
\begin{align*}
\Psi_0^{-1}\big(v_{\Psi_0}(f+g)\big)&\le\liminf_{n\ua\infty}\big( \Psi_0^{-1}\big(v_{\Psi_0}(f,\bS_n)\big)
+ \Psi_0^{-1}\big(v_{\Psi_0}(g,\bS_n)\big)\big)\le   \Psi_0^{-1}\big(v_{\Psi_0}(f)\big).
\end{align*}
In the same manner, we get $v_{\Psi_0}(f+g)\ge v_{\Psi_0}(f)$ by taking a sequence of partitions $(\bS_n)_{n\in\bN}$ with $|\bS_n|\to0$ and $v_{\Psi_0}(f,\bS_n)\to v_{\Psi_0}(f)$ and applying the left-hand side of \eqref{Mulholland ineq}.
\end{proof}

\begin{proof}[Proof of \Cref{semimartingale thm} for $K<H$] We only need to consider the case in which $X$ admits a finite and nontrivial quadratic variation, which by \Cref{main thm} and our assumption $K<H$ is tantamount to $1/2=H\wedge K=K$.
 We assume by way of contradiction that $X$ can be decomposed as $X= M+ A$, where $  M$ is a continuous local martingale with $M_0=0$ and $  A$ is a process whose sample paths are of bounded variation on $[0,1]$. By \Cref{main thm}, $\<X\>_t=\<M\>_t =Vt$ for a random variable $V>0$. Since $X$ is Gaussian, Stricker's theorem \cite{Stricker} implies that $M$ is Gaussian and thus has independent increments. Hence, $V$ must be equal to a  constant $c>0$. Therefore, $ B:=c^{-1/2} M$ is a Brownian motion by L\'evy's theorem. In the formulation of 
 Corollary 12.24 in \cite{dudley2011concrete}, a theorem by Taylor states that $v_\Psi(B)=1$ $\bP$-a.s. Since $v_\Psi(c^{-1/2}A)=0$ 
$\bP$-a.s.~by \Cref{Phi var lemma} (a), we must have that $v_\Psi(c^{-1/2}X)=1$ by  \Cref{Phi var lemma} (c). However, the sample paths of $X$ are H\"older continuous with exponent $1/2$ according to \Cref{Hoelder prop}, which is a contradiction to  \Cref{Phi var lemma} (b).
\end{proof}

\begin{proof}[Proof of \Cref{semimartingale thm} for $H<K$] 
In this case, we have $H=1/2$. We let $(\cF_t)_{t\ge0}$ denote the natural filtration of $X$ and $\Lambda^n_j=X(jb^{-n})-X((j-1)b^{-n})$. Our goal is to prove that there is a constant $\lambda>0$ such that for all sufficiently large $n$,
\begin{align*}S_n:=\sum_{j=0}^{b^n-1}\bE\Big[\bE[\Lambda^n_{j+1}|\cF_{t_j}]^2\Big]\geq \lambda.
\end{align*}
This will imply that $X$ is not a quasi-Dirichlet process in the sense of  \cite[Definition 3]{RussoTudor} and hence not a semimartingale (see the proof of \cite[Proposition 6]{RussoTudor} for details). To this end, note first that by Jensen's  inequality for conditional expectations and for $j\ge (b-1)b^{n-1}$, 
$$\bE\Big[\bE[\Lambda^n_{j+1}|\cF_{t_j}]^2\Big]\ge \bE\Big[\bE[\Lambda^n_{j+1}|\Lambda^n_{j+1-b^{n-1}}]^2\Big].
$$
Next, since $(\Lambda^n_{j+1},\Lambda^n_{j+1-b^{n-1}})$ is a centered Gaussian vector, the conditional expectation on the right-hand side is given as follows,
$$\bE[\Lambda^n_{j+1}|\Lambda^n_{j+1-b^{n-1}}]=\frac{\bE[\Lambda^n_{j+1}\Lambda^n_{j+1-b^{n-1}}]}{\bE[(\Lambda^n_{j+1-b^{n-1}})^2]}\Lambda^n_{j+1-b^{n-1}}.
$$
We will show in \Cref{covariance lower bound lemma} that $\bE[\Lambda^n_{j+1}\Lambda^n_{j+1-b^{n-1}}]\ge \lambda_1b^{-n}$ for some constant $\lambda_1>0$. Moreover, \Cref{MC} will show that the denominator is bounded by $L_1b^{-n}$ for another constant $L_1$. 
Hence, 
$$S_n\ge \sum_{j=(b-1)b^{n-1}}^{b^n-1}\frac{\bE[\Lambda^n_{j+1}\Lambda^n_{j+1-b^{n-1}}]^2}{\bE[(\Lambda^n_{j+1-b^{n-1}})^2]}\ge \sum_{j=(b-1)b^{n-1}}^{b^n-1}\frac{\lambda_1b^{-n}}{L_1} ,
$$
which is bounded from below by $\lambda:=\lambda_1/L_1$.
\end{proof}

The following lemma shows in particular that the Wiener--Weierstra\ss\ bridge with $H=1/2<K$ is, at least locally, a quasi-helix in the sense of Kahane \cite{Kahane1981,kahane1993some}. Analogous estimates will be derived in the more general case $H\leq K$ in an upcoming work, but $H=1/2<K$ is all we need here.

\begin{lemma} \label{MC}
Let $ X$ be the Wiener--Weierstra\ss\ bridge  with $H=1/2<K$. 
\begin{enumerate}
\item There exists a constant $L>0$ such that for all $s,t\in[0,1]$,
$$ \bE[(X(t)-X(s))^2]\leq L|t-s|.$$
\item For each $\lambda\in(0,1)$ there exists $\eps>0$ such that for $s,t\in[0,1]$ with $|t-s|<\eps$,
$$ \bE[(X(t)-X(s))^2]\ge\lambda|t-s|.
$$
\end{enumerate}
\end{lemma}

\begin{proof} (a) There is a constant $L_0$ such that $ \bE[|B(\fp{b^nt})-B(\fp{b^ns})|^2]\le L_0 b^n|t-s|$, due to \eqref{fbb} and the fact that $\kappa$ is H\"older continuous with exponent $\kappa>1/2$. Then one uses the definition \eqref{WW def eq} of the Wiener--Weierstra\ss\ bridge and Minkowski's inequality to obtain
$$ \bE[(X(t)-X(s))^2]^{1/2}\le\sum_{n=0}^\infty \alpha^n  \bE[|B(\fp{b^nt})-B(\fp{b^ns})|^2]^{1/2}\le \sqrt{L_0|t-s|}\sum_{n=0}^\infty (\alpha b^{1/2})^n.
$$
Since, by assumption, $\alpha^2b<1$, (a) follows.

(b) We let $c$ be the H\"older constant of $\kappa$, i.e., $|\kappa(r)-\kappa(u)|\le c|r-u|^\tau$ for all $r,u\in[0,1]$. Then we make the following definitions for $M\in\bN$. 
$$K_1(M):=(1+2\sup|\kappa|)\sum_{m=M}^\infty\alpha^m,\qquad K_2(M):=\begin{cases}
\frac{2c\alpha^M}{\alpha b^\tau-1}&\text{for $\alpha b^\tau>1$,}\\
2cMb^{-\tau M}&\text{for $\alpha b^\tau=1$,}\\
\frac{2cb^{-\tau M}}{1-\alpha b^\tau}&\text{for $\alpha b^\tau<1$.}\\
\end{cases}
$$
Then we choose $L$ be such that $K_1(M)+K_2(M)<1-\sqrt\lambda$ for all $M\ge L$ and set $\eps:=b^{-L}$. Then we fix $0\leq s< t\leq 1$ with $|s-t|\leq \eps$.
 Let $M:=\lfloor{-\log_b}(t-s)\rfloor$ so that $b^{-M-1}<t-s\leq b^{-M}$ and  $M\ge L$. As in the proof of Proposition \ref{Gaussian martingale bridge prop}, we write
\begin{align*}
 X(t)- X(s)&=\sum_{m=0}^\infty \alpha^m(W(\fp{b^mt})-W(\fp{b^ms})-(\kk(\fp{b^mt})-\kk(\fp{b^ms}))W(1))=\int_0^1g(x)\,dW(x)
\end{align*}
as a Wiener integral, where\begin{align*}
g(x):=\sum_{m=0}^{\infty}\alpha^m\bone_{[\fp{b^ms},\fp{b^mt}]}(x)-\sum_{m=0}^\infty \alpha^m (\kk(\fp{b^mt})-\kk(\fp{b^ms})).
\end{align*}
Here we use again the convention that for $x<y$, the indicator function $\Ind{[y,x]}$ is defined as $-\Ind{[x,y]}$.
 Define
\begin{align*}
\ell:=\inf\big\{1\leq k\leq M-1:\fp{b^ks}>\fp{b^kt}\big\}\wedge M.
\end{align*}
 We claim that for $0\leq k<\ell$,
\begin{align}
0\leq \fp{b^ks}<\fp{b^kt}<1\ \text{ and }\ \fp{b^kt}-\fp{b^ks}=b^kt-b^ks\leq b^{k-M},\label{k1}
\end{align}
and for $\ell\leq k<M$,
\begin{align}
0\leq \fp{b^kt}<\fp{b^ks}<1\ \text{ and }\ \fp{b^ks}-\fp{b^kt}= 1-(b^kt-b^ks)\geq 1-b^{k-M}.\label{k2}
\end{align}
These assertions are obvious in case $\ell=M$. For $\ell<M$, we have  $0< b^\ell t-b^\ell s\leq b^{\ell-M}<1$.  Together with $\fp{b^\ell s}>\fp{b^\ell t}$, this implies $\fp{b^\ell  t}+(1-\fp{b^\ell  s})\leq b^{\ell-M}$. It follows that, for $k\in[\ell,M)$, we have 
$\fp{b^kt}=b^{k-\ell}\fp{b^\ell t}\le b^{k-\ell}b^{\ell-M}\le b^{-1}$ and $1-\fp{b^ks}\le b^{k-\ell }(1-\fp{b^\ell s})\leq b^{-1}$. Therefore, $\fp{b^kt}<\fp{b^ks}$, i.e., the order of $\fp{b^kt}$ and $\fp{b^ks}$ flips at most once for $0\leq k<M$ (namely when $k=\ell $) and after they flip, one of them stay close to $0$ and the other one close to $1$. It is thus clear that their distance before the flip must be $b^k(t-s)$, whereas after the flip it is $1-b^k(t-s)$. This proves \eqref{k1} and \eqref{k2}.

According to \eqref{k2}, we have for $\ell\le k<M$ that $\bone_{[\fp{b^ms},\fp{b^mt}]}=1-\bone_{ [0,\fp{b^mt}]\cup[\fp{b^ms},1]}$.
Hence, we may write $g(x)=\sum_{i=1}^3g_i(x)$ where
\begin{align*}
g_1(x)&:=\sum_{m=0}^{\ell -1}\alpha^m \bone_{[\fp{b^ms},\fp{b^mt}]}(x)+\sum_{m=\ell }^{M-1}\alpha^m\bone_{ [0,\fp{b^mt}]\cup[\fp{b^ms},1]}(x),\\
g_2(x)&:=-\displaystyle\sum_{m=0}^{\ell -1}\alpha^m(\kk(\fp{b^mt})-\kk(\fp{b^ms}))-\sum_{m=\ell }^{M-1}\alpha^m\left(1-(\kk(\fp{b^ms})-\kk(\fp{b^mt}))\right),\\
g_3(x)&:=\displaystyle\sum_{m=M}^\infty \alpha^m \bone_{[\fp{b^ms},\fp{b^mt}]}(x)-\sum_{m=M}^\infty \alpha^m (\kk(\fp{b^mt})-\kk(\fp{b^ms})).
\end{align*}
Note that $g_2(x)$ does not depend on $x$.

Clearly, $g_1$ is bounded from below by the term corresponding to $m=0$, i.e., 
\begin{equation}\label{g1 lower est}
g_1\ge \Ind{[s,t]}.
\end{equation}
 We also have
$g_3(x)\geq -K_1(M)
$ for all $x\in[0,1]$.
By (\ref{k2}), for $\ell \leq k<M$, $\fp{b^ks}-\fp{b^kt}\geq 1-b^{k-M}$, so that by H\"{o}lder continuity of $\kk$,  for all $x\in[0,1]$,
\begin{align*}
g_2(x)&\geq -\sum_{m=\ell }^{M-1}\alpha^m (\fp{b^mt}^\tau+(1-\fp{b^ms})^\tau)-\sum_{m=0}^{\ell -1}\alpha^m(b^m(t-s))^{\tau}\geq -2b^{-M\tau}\sum_{m=0}^{M-1}(\alpha b^\tau)^m\geq-K_2(M).
\end{align*}
Therefore, $
g(x)\geq 1+g_2(x)+g_3(x)\geq 1-(K_1(M)+K_2(M))\geq \sqrt{\lambda}$ for $x\in[s,t]$.
Now the It\^o isometry gives
\[
\bE[(X(t)- X(s))^2]=\bE\bigg[\bigg(\int_0^1g(r)\,dW(r)\bigg)^2\bigg]=
 \int_0^1g^{2}(r)\,dr\ge  \int_s^tg^{2}(r)\,dr\geq \lambda|t-s|,
\]
as required. 
\end{proof}

\begin{lemma}\label{covariance lower bound lemma} 
For $H=1/2<K$  there exist  $\lambda>0$ and $M\in\bN$ such that for all $n\ge M$ and $b^{n-1}\le j<b^n$,
\begin{equation}\label{covariance lower bound lemma eq}\bE\Big[\Big(X\big((j+1)b^{-n}\big)-X\big(jb^{-n}\big)\Big)\Big(X\big((j+1)b^{-n}-b^{-1}\big)-X\big(jb^{-n}-b^{-1}\big)\Big)\Big]\ge \lambda b^{-n}.
\end{equation}
\end{lemma}

\begin{proof}For any $t$ of the form $t=ib^{-n}$ 
we have
\begin{equation}\label{covariance lower bound lemma X decomposition eq} 
X(t)=B(t)+\sum_{m=1}^n\alpha^mB(\fp{tb^{m}})=:B(t)+\wt X(t).
\end{equation}
Substituting all occurrences of $X$ in \eqref{covariance lower bound lemma eq} with \eqref{covariance lower bound lemma X decomposition eq}  and factoring out the product yield four separate terms, which we are going to analyze individually in the sequel. Recall that $B(t)=W(t)-\kappa(t)W(1)$, where $\kappa$ is H\"older continuous with exponent $\tau>H=1/2$. We also use the shorthand notation $t_4:=(j+1)b^{-n}$, $t_3:=jb^{-n}$, $t_2:=(j+1)b^{-n}-b^{-1}$, and $t_1:=jb^{-n}-b^{-1}$. 

First, we analyze the term 
\begin{align*}
\lefteqn{\Big|\bE\Big[\big(B(t_4)-B(t_3)\big)\big(B(t_2)-B(t_1)\big)\Big]\Big|}\\
&=\Big|(\kappa(t_4)-\kappa(t_3))(\kappa(t_2)-\kappa(t_1))-(\kappa(t_4)-\kappa(t_3))(t_2-t_1)-(\kappa(t_2)-\kappa(t_1))(t_4-t_3)\Big|\\
&\le L b^{-2\tau n}=o(b^{-n}).
\end{align*}

Next, we analyze the mixed terms. To this end, note that for $m\ge1$ we have $\fp{t_1b^m}=\fp{jb^{m-n}-b^{m-1}}=\fp{jb^{m-n}}=\fp{t_3b^m}$. In the same way, we have $\fp{t_2b^m}=\fp{t_4b^m}$.
Thus, 
\begin{equation}\label{wt X}
\wt X(t_2)-\wt X(t_1)=\wt X(t_4)-\wt X(t_3).
\end{equation}
Hence, 
the first mixed term is
\begin{equation*}
\begin{split}
\bE\Big[\big(B(t_4)-B(t_3)\big)\big(\wt X(t_2)-\wt X(t_1))\big)\Big]&=\bE\Big[\big(B(t_4)-B(t_3)\big)\big(\wt X(t_4)-\wt X(t_3))\big)\Big]\\
&=\sum_{m=1}^n\alpha^m\bE\Big[\big(B(t_4)-B(t_3)\big)\big(B(\fp{t_4b^m})-B(\fp{t_3b^m})\big)\Big]\\
&=\sum_{m=1}^n\alpha^m
\int_0^1f(x)g_m(x)\,dx,
\end{split}
\end{equation*}
where 
\begin{align*}
f&=\Ind{[t_3,t_4]}-\big(\kappa(t_4)-\kappa(t_3)\big)=:f_1-C_0,\qquad
g_m=\Ind{[\fp{t_3b^m},\fp{t_4b^m}]}-\big(\kappa(\fp{t_4b^m})-\kappa(\fp{t_3b^m})\big),
\end{align*}
where we use again the convention that $\Ind{[b,a]}:=-\Ind{[a,b]}$ if $a<b$.

Now we claim that $g_m$ can be written as $g_m(x)=\Ind{I_m}(x)+C_m$, where $I_m$ is a subset of $[0,1]$ of total length $b^{m-n}$ and $C_m$ is a constant with  $|C_m|\le L_1b^{(m-n)\tau}$ for   another constant $L_1>0$. Indeed, if $\fp{t_3b^m}\le \fp{t_4b^m}$, we can take $I_m:=[\fp{t_3b^m},\fp{t_4b^m}]$,  and $C_m:=-\kappa(\fp{t_4b^m})+\kappa(\fp{t_3b^m})$. Then $C_m$ satisfies $|C_m|\le L_1b^{(m-n)\tau}$  due to the H\"older continuity of $\kappa$. If $\fp{t_3b^m}> \fp{t_4b^m}$, then $\Ind{[\fp{t_3b^m},\fp{t_4b^m}]}=\Ind{I_m}-1$ for $I_m:=[0,\fp{t_4b^m}]\cup[\fp{t_3b^m},1]$. Hence, we can let 
$$C_m:=-\kappa(\fp{t_4b^m})+\kappa(\fp{t_3b^m})-1=\kappa(0)-\kappa(\fp{t_4b^m})+\kappa(\fp{t_3b^m})-\kappa(1).$$
 Also in this case, $|C_m|\le L_1b^{(m-n)\tau}$ by the H\"older continuity of $\kappa$.
 
 It follows that 
 $$\int_0^1fg_m\,dx=\int_{I_m}f_1\,dx+C_m\int_0^1f_1\,dx-C_0|I_m|-C_0C_m,$$
 where $|I_m|$ denotes the total length of $I_m$. Observe first that $ \int_{I_m}f_1\,dx\ge0$. Next, we have 
 $|C_m\int_0^1f_1\,dx|\le L_1b^{(m-n)\tau-n}$, $|C_0|I_m||\le cb^{-(\tau+1)n}$, where $c$ is the H\"older constant of $\kappa$, and $|C_0C_m|\le cL_1b^{(m-2n)\tau}$. Altogether, this gives
 $\int_0^1fg_m\,dx\ge-L_2b^{(m-2n)\tau}
 $
 for some constant $L_2$. We conclude that 
 $$\sum_{m=1}^n\int_0^1fg_m\,dx\ge-L_2\sum_{m=1}^n\alpha^m b^{(m-2n)\tau}.$$
 Using  that $\tau>1/2$ and $\alpha^2b<1$ one checks that the right-hand side  is of the order $o(b^{-n})$. The second mixed term is handled in the same manner. 
 
 Finally, we analyze the term
 $$\bE\Big[\big(\wt X(t_4)-\wt X(t_3)\big)\big(\wt X(t_2)-\wt X(t_1)\big)\Big]=\bE\Big[\big(\wt X(t_4)-\wt X(t_4)\big)^2\Big],
 $$
 where we have used \eqref{wt X}. To this end, we proceed as in the proof of \Cref{MC} (b) with $t:=t_4=(j+1)b^{-n}$ and $s:=t_3=jb^{-n}$. We retain the notation from that proof with the only difference that we sum from $m=1$ instead of $m=0$. The estimates for $g_2$ and $g_3$ obtained in the final paragraph of that proof remain true, but \eqref{g1 lower est} is no longer valid, because it was obtained by looking at the case $m=0$. However, if $\ell>1$, then we estimate $g_1\ge\alpha\Ind{[\fp{bs},\fp{bt}]}$, and the length of the interval $[\fp{bs},\fp{bt}]$ is  equal to $b^{1-n}$. If $\ell=1$, then there is an integer $k$ such that $jb^{1-n}=bs<k\le bt=(j+1)b^{1-n}$. Hence, $g_1\ge\alpha \Ind{[0,\fp{bt}]\cup[\fp{bs},1]}=\alpha \Ind{[\fp{bs},1]}$, and the length of the interval $[\fp{bs},1]$ is also equal to $b^{1-n}$. As in the proof of  \Cref{MC} (b)  we now get that there is a constant $\lambda>0$ such that $\bE[(\wt X(t_4)-\wt X(t_4))^2]\ge\lambda b^{-n}$.
 
 Putting everything together yields that \eqref{covariance lower bound lemma eq} holds for all sufficiently large $n$.
  \end{proof}

\subsection{Other proofs}

\begin{proof}[Proof of  \Cref{covthm}]
 For fixed $s\in(0,1)$, let 
\begin{equation}\label{phi for c eq}
\phi(t):=c(s,t)-\alpha c(s,\fp{bt}),\qquad t\in[0,1].
\end{equation}
Since $c(s,t)$ is uniformly bounded, we obtain the representation
\begin{equation}\label{c phi rep}
c(s,t)=\sum_{m=0}^\infty \alpha^m\phi(\fp{b^mt}).
\end{equation}
Our goal is to apply  \Cref{p variation prop}. To this end, we note first that $\phi$ satisfies $\phi(0)=0=\phi(1)$. Moreover, we get from \eqref{phi for c eq}
 that
\begin{align*}
\phi(t)&=\sum_{m=0}^\infty \sum_{n=0}^\infty \alpha^{n+m}\bE\big[B( \fp{b^m s} )B( \fp{b^n t})\big]-\sum_{m=0}^\infty \sum_{n=0}^\infty \alpha^{n+m+1}\bE\big[B( \fp{b^m s})B(  \fp{b^{n+1} t})\big]\\
&=\sum_{m=0}^\infty \alpha^m\bE\big[B( \fp{b^m s})B(t)\big].
\end{align*}
Since $\kappa$ is given by \eqref{standard kappa}, one checks that $\phi$ is H\"older continuous with exponent $2H>K$.  Therefore,  \Cref{p variation prop} applies to the function $t\mapsto c(s,t)$ in \eqref{c phi rep} and we conclude that it has finite linear $(1/K)$-th variation. 
Moreover, for each fixed $u\in(0,1)$, the function $t\mapsto\bE\big[B(u)B(t)\big]$ is nonnegative and null only for $t\in\{0,1\}$. Hence, condition \eqref{cond} is satisfied and the $(1/K)$-th variation is nontrivial.
 \end{proof}

\begin{proof}[Proof of \Cref{co}] Part (a) follows  from Theorem \ref{main thm} and Proposition \ref{covthm} provided that $H> K$.

To prove  (b), consider a centered Gaussian process $(Y_t)_{t\in[0,1]}$ with covariance function $c(s,t)$. We denote by $\bT_n:=\{kb^{-n}:k=0,\dots, b^n\}$, $n\in\bN$, the $b$-adic partitions. For $t\in\bT_n$, we let $t':=\inf\{u\in\bT_n:u>t\}\wedge1$ denote the successor of $t$ in $\bT_n$.
 H\"{o}lder's inequality gives 
\begin{align}
\sum_{t\in\bT_n}|c(s,t')-c(s,t)|^p&\leq\sum_{t\in\bT_n}\left(\bE[|Y(s)(Y(t')-Y(t))|]\right)^p\nonumber\\
&\leq \sum_{t\in\bT_n}\left(\left(\bE[|Y(s)|^q]\right)^{1/q}\left(\bE[|Y(t')-Y(t)|^p]\right)^{1/p}\right)^p\nonumber\\
&=\left(\bE[|Y(s)|^q]\right)^{p/q}\bE\bigg[\sum_{t\in\bT_n}|Y(t')-Y(t)|^p\bigg].\label{Y exp pth var}
\end{align} 
Suppose by way of contradiction that, for some $s\in[0,1]$, the expression on the left-hand side of \eqref{cov pth var eq}
 is infinite. Then obviously $\bP(Y(s)=0)<1$ and by passing to the limit $n\ua\infty$ in \eqref{Y exp pth var}, we get
 \begin{align}
     \limsup_{n\ua\infty}\bE\bigg[\sum_{t\in\bT_n}|Y(t')-Y(t)|^p\bigg]=\infty.\label{eq:weget}
 \end{align}
   But since $Y$ is a Gaussian process, 
an application of Fernique's theorem (Theorem 1.3.2 in~\cite{fernique1975regularite} or Lemma 2.10 in \cite{dudley2014uniform}) applied to the  seminorm
$$
N(Y)=\sup_{n\in\bN}\bigg(\sum_{t\in\bT_n}\bE[|Y(t')-Y(t)|^p]\bigg)^{1/p}
$$ yields that the pathwise  $p$-th variation of $Y$ cannot be $\bP$-a.s.~finite.
 This is a contradiction to \eqref{eq:weget}. \end{proof}

\begin{proof}[Proof of Proposition  \ref{Gaussian martingale bridge prop}] We may assume without loss of generality that $1=\<M\>_1=\int_0^1\varphi(s)\,ds$. As in \Cref{fractional bridge thm lemma}, let $(r_m)_{m\in\bN}$ be an arbitrary sequence of integers such that $r_m\in\{0,\dots,b^m-1\}$ and assume in addition that the set $\{\frac{r_m}{b^m}:m\in\mathbb N\}$ is dense in $[0,1]$. Then
\begin{equation}\label{finite n Wiener integral rep eq}
\begin{split}
\lefteqn{\sum_{m=1}^n\alpha^{-m}\Big(B\Big(\frac{r_m+1}{b^m}\Big)-B\Big(\frac{r_m}{b^m}\Big)\Big)}\\
&=\sum_{m=1}^n\alpha^{-m}\Big(M\Big(\frac{r_m+1}{b^m}\Big)-M\Big(\frac{r_m}{b^m}\Big)\Big)-\sum_{m=1}^n \alpha^{-m}\bigg(\int_{\frac{r_m}{b^m}}^{\frac{r_m+1}{b^m}}\varphi(s)\,ds\bigg)M(1).
\end{split}
\end{equation}
Let $L>0$ be such that $0\le\varphi(t)\le L$. Then, 
 $$C:=\sum_{m=1}^\infty \alpha^{-m}\int_{\frac{r_m}{b^m}}^{\frac{r_m+1}{b^m}}\varphi(s)\,ds\leq \sum_{m=1}^\infty \alpha^{-m}b^{-m}L<\infty.$$ 
So the function 
\begin{align*}
f (t):=\sum_{m=1}^\infty\alpha^{-m}\Ind{[\frac{r_m}{b^m},\frac{r_m+1}{b^m}]}(t) -C
\end{align*}
is well-defined a.e.~on $[0,1]$, and one  checks as in \eqref{fH in L1H eq} that  $ f \in L^2[0,1]$. It hence follows from \eqref{finite n Wiener integral rep eq}
that
\begin{equation*}
\sum_{m=1}^\infty\alpha^{-m}\Big(B\Big(\frac{r_m+1}{b^m}\Big)-B\Big(\frac{r_m}{b^m}\Big)\Big)=\int_0^1f(t)\,dM(t).
\end{equation*}
Thus, the second moment of the left-hand expression is finite and given by
\begin{equation}\label{int f2phidt eq}
\mathbb E\bigg[\bigg(\int_0^1f(t)\,dM(t)\bigg)^2\bigg] =\int_0^1(f(t))^2\,d\<M\>_t\ge \int_I(f(t))^2\varphi(t)\,dt,
\end{equation}
where $I$ is the  nonempty open interval on which $\varphi>0$ by assumption. Since there are infinitely many $m\in\mathbb N$ for which $\frac{r_m}{b^m}\in I$, we see as in \eqref{fH>0 eq} that the rightmost integral in \eqref{int f2phidt eq} must be strictly positive.

Next, whenever $\gamma<1/2$ is given, then the sample paths of $M$ are $\mathbb P$-a.s.~H\"older continuous with exponent  $\gamma$, because $M$ has the same law as $W(\<M\>)$ for some standard Brownian motion $W$; this follows from the standard DDS time change argument (e.g., Theorem V.1.6 in~\cite{RevuzYor}).  

Our assertion now follows as in the proof of \Cref{main thm} (a), once we have checked that the random variables $R_m$ defined  in \eqref{R and X} are such that  $\{\frac{R_m}{b^m}:m\in\mathbb N\}$ is $\mathbb P$-a.s.~dense in $[0,1]$. This will follow from a standard Borel-Cantelli argument. Indeed, fix a nonempty open set $J\subseteq[0,1]$ and choose a subinterval $[kb^{-N},(k+1)b^{-N})\subseteq J$ where $k,N\in\bN_0$. Write $kb^{-N}=\sum_{i=1}^Nk_ib^{i-1-N}$ where $k_i\in\{0,\dots,b-1\}$. Then for every fixed $m\geq N$, apart from null sets we have$$\left\{\frac{R_m}{b^m}\in[kb^{-N},(k+1)b^{-N})\right\}=\bigcap_{i=1}^N\{U_{i+m-N}=k_i\}.$$
Therefore, the events
$$
\left\{\frac{R_{\ell N}}{b^{\ell N}}\in[kb^{-N},(k+1)b^{-N})\right\},\qquad \ell\in\bN,
$$
are independent. Obviously, $\bP(U_{i+\ell N-N}=k_i,\ 1\leq i\leq N)=b^{-N}$ for each $\ell\in\bN$, so the second Borel-Cantelli lemma finishes the argument. 
\end{proof}

\appendix

\section{Fractal functions with H\"older continuous base}\label{Hoelder section}

In this appendix, we collect some preliminary results needed in   \Cref{main thm} (a) and \Cref{Gaussian martingale bridge prop}. In these results,  the parameter $K$ resulting from the Weierstra\ss-type convolution is smaller than the Hurst parameter $H$ of the underlying Gaussian bridge. It turns out that  this particular case can be analyzed to some degree by extending techniques that were developed for the study of deterministic fractal functions of the form  
\begin{equation}\label{f  eq}
f(t)=\sum_{n=0}^\infty \alpha^n\phi(\fp{b^n t}),\end{equation}
where  $\alpha\in(0,1)$, $b\in\{2,3,\dots\}$, and $\phi:[0,1]\to\bR$ is a   continuous  function  with $\phi(0)=\phi(1)$. As mentioned in the introduction and \Cref{resultssection}, the functions of this type include the Weierstra\ss\ and Takagi--Landsberg functions, but in the existing literature, their study was mainly restricted to the case in which $\phi$ is Lipschitz continuous; see, e.g.,~\cite{BaranskiSurvey} and the references therein. In our application to Gaussian Weierstra\ss\ bridges, $\phi$ will be a typical sample path of fractional Brownian bridge or a more general Gaussian bridge, and so the  Lipschitz condition does not apply. In this section, we therefore discuss the case in which $\phi$ is H\"older continuous with exponent $\gamma\in(0,1]$. In the application to the proofs of  \Cref{main thm} (a) and \Cref{Gaussian martingale bridge prop} we will actually have $\gamma>K$. Although the main purpose of this section is to prepare for the proofs of our results on Gaussian Weierstra\ss\ bridges, we believe that it could also be of independent interest to the study of deterministic functions $f$ of the form \eqref{f  eq}.

\begin{proposition}\label{Hoelder prop} Suppose that $\phi$ is H\"older continuous with exponent $\gamma\in(0,1]$ and let $K=(-\log_b\alpha)\wedge 1 $. \begin{enumerate}
\item\label{Hoelder prop (b)}  If $ K \neq\gamma$, then $f$ is H\"older continuous with exponent $ K \wedge\gamma$.
\item\label{Hoelder prop (c)}  If $ K =\gamma$, then there exists a constant $c>0$ such that  $w(t):=c t^\gamma\log t^{-1}$ is a (uniform) modulus of continuity for $f$. In particular, $f$ is H\"older continuous for every exponent $\beta<\gamma$.
\end{enumerate}
\end{proposition}

\begin{proof} Consider the periodic extension of $\phi$ to all of $\bR$, and denote this function again by $\phi$. Using the elementary inequality $a^\gamma+b^\gamma\le 2^{1-\gamma}(a+b)^\gamma$, which holds for $a,b\ge0$ and $0<\gamma\le1$, one checks that the periodic extension $\phi$ is also H\"older continuous with exponent $\gamma$ on all of $\bR$. So let $C$ be such that $|\phi(x)-\phi(y)|\le C|x-y|^\gamma$ for all $x,y\ge0$.   Throughout this proof, we also consider the periodic extension of $f$ and denote it again by $f$. 

In case $ K >\gamma$, we have $\alpha  b^\gamma<1$ and so, for $t,s\in\bR$,
$$
|f(t)-f(s)|\leq \sum_{n=0}^{\infty}\alpha ^n|\phi(b^nt)-\phi(b^ns)|\leq C |t-s|^\gamma\sum_{n=0}^{\infty}(\alpha  b^\gamma)^n=L |t-s|^\gamma
$$
for a constant $L$. This proves that $f$ is H\"older continuous with exponent $\gamma$.

For $ K \le\gamma$,   let $s,t\in\bR$ be given. Due to the periodicity of $f$, we may assume without loss of generality that $|t-s|\le1/2$.  We choose $N\in\bN$ such that $b^{-N}<|t-s|\le b^{1-N}$. Then we have
\begin{equation}\label{Hoelder prop eq}
\begin{split}
|f(t)-f(s)|&\leq \sum_{n=0}^{N-1}\alpha ^n|\phi(b^nt)-\phi(b^ns)|+\sum_{n=N}^{\infty}\alpha ^n|\phi(b^nt)-\phi(b^ns)|\\
&\leq C|t-s|^\gamma\sum_{n=0}^{N-1}(\alpha  b^{\gamma})^n+2\sup_{x\in[0,1]}|\phi(x)|\frac{\alpha ^N}{1-\alpha }.\end{split}
\end{equation}
If $ K <\gamma$, we have $\alpha  b^\gamma>1$ and get from \eqref{Hoelder prop eq} that 
\begin{align*}
|f(t)-f(s)|&\leq C|t-s|^\gamma\frac{(\alpha  b^\gamma)^N-1}{\alpha  b^\gamma-1}+2\sup_{x\in[0,1]}|\phi(x)|\frac{\alpha ^N}{1-\alpha }\le \Big(\frac{Cb^\gamma}{\alpha  b^\gamma-1}+\frac{2\sup| \phi|}{1-\alpha }\Big)\alpha ^N.
\end{align*}
Since $\alpha ^N=b^{- K  N}\le |t-s|^ K $, our proof of part \ref{Hoelder prop (b)} is complete.

In case \ref{Hoelder prop (c)}, we have $\alpha  b^\gamma=1$ and get from \eqref{Hoelder prop eq} that 
\begin{align*}
|f(t)-f(s)|&<CN|t-s|^\gamma+\frac{2\sup |\phi|}{1-\alpha }\alpha ^N\leq C\bigg(1+\frac{\log|t-s|^{-1}}{\log b}\bigg)|t-s|^{\gamma}+\frac{2\sup |\phi|}{1-\alpha }|t-s|^\gamma,
\end{align*}
and this is less than or equal to $c|t-s|^\gamma\log |t-s|^{-1}$ for $|t-s|\le1/2$.\end{proof}

The following result  can be proved in the same way as Theorem 2.1  and Proposition 2.4 in~\cite{SchiedZZhang}, where the key is the representation \eqref{Vn Rm eq}. 

\begin{proposition}\label{p variation prop}Suppose that $\phi$ is H\"older continuous with exponent $\gamma\in(0,1]$ and  that $b\in\{2,3,\dots\}$  and $\alpha\in(0,1)$ are such that $\alpha  b^\gamma>1$. Then
\begin{equation}\label{Z}
Z:=\sum_{m=1}^\infty \alpha ^{-m}\Big(\phi\big((R_m+1)b^{-m}\big)-\phi\big(R_mb^{-m}\big)\Big)
\end{equation}
is a bounded random variable, and for $p:=-\log_{\alpha }b$,
\begin{align*}
\<f\>^{(p)}_t=
\lim_{n\uparrow\infty}\sum_{k=0}^{\lfloor tb^n\rfloor}\big|f((k+1)b^{-n})-f(kb^{-n})\big|^p=
t\cdot \mathbb{E}_R[|Z|^{p}],\qquad t\in[0,1].
\end{align*}
Moreover, $\mathbb{E}_R[|Z|^{p}]>0$ as soon as 
\begin{equation}\label{cond}
\{0\}\neq \{\phi(b^{-k}):k\in\mathbb{N}\}\subset[0,\infty).
\end{equation}
\end{proposition}

\begin{remark}By considering $\widetilde\phi(t):=-\phi(t)$ or $\widehat\phi(t):=\phi(-t)$ or $\overline\phi(t):=-\phi(-t)$, one sees that~\eqref{cond} can be replaced by  several similar conditions. For instance, requiring~\eqref{cond}  for $\overline\phi$ is equivalent to the condition   $\{0\}\neq \{\phi(1-b^{-k}):k\in\mathbb{N}\}\subset(-\infty,0]$. 
 \end{remark}

\begin{definition}Let $\phi:[0,1]\to\bR$ be H\"older continuous with exponent $\gamma\in(0,1]$ and $\phi(0)=\phi(1)$, $b\in \{2,3,\dots\}$, and $\alpha b^\gamma>1$. The  function $\phi$ is called a  \emph{valid base function for $b$ and $\alpha$} if the random variable $Z$ in \eqref{Z} is not $\bP_R$-a.s.~null. 
\end{definition}

The following proposition shows that fractal functions of the form \eqref{f intro eq} are often themselves  valid base functions. 

\begin{proposition}\label{kernel kernel prop} Suppose that $\phi:[0,1]\to\mathbb R$ is H\"older continuous with exponent $\gamma\in(0,1]$ and a valid base function for $b\in\{2,3,\dots\}$ and $\alpha\in(b^{-\gamma},1)$. Then, if $0<\beta<1/b^\gamma$,\begin{equation}\label{kernel kernel psi eq}
\psi(t):=\sum_{m=0}^\infty\beta^m\phi(\fp{b^mt})
\end{equation}
is a valid base function for $b$ and $\alpha$.
\end{proposition}

\begin{proof}First, it follows from \Cref{Hoelder prop} (b) that $\psi$ is H\"older continuous with exponent $\gamma$, and so we may apply \Cref{p variation prop}. 
Let
\begin{equation}\label{kernel kernel Z eq} 
Z:=\sum_{m=1}^\infty \alpha ^{-m}\Big(\psi\big((R_m+1)b^{-m}\big)-\psi\big(R_mb^{-m}\big)\Big),
\end{equation}
where the $R_m$ are as in \eqref{R and X}. As in the proof of \Cref{Hoelder prop}, we extend $\phi$ to all of $\bR$ by periodicity.  Then, for any $x$ and $\ell\le m$, 
$$\phi\big(x+R_mb^{-\ell}\big)=\phi\Big(x+\sum_{i=1}^mU_ib^{i-1-\ell}\Big)=\phi\Big(x+\sum_{i=1}^\ell U_ib^{i-1-\ell}\Big)=\phi\big(x+R_\ell b^{-\ell}\big).
$$
Using this fact and once again the periodicity of $\phi$, we get
\begin{align}
Z&=\sum_{m=1}^\infty \alpha ^{-m}\sum_{n=0}^\infty \beta^n\Big(\phi\big((R_m+1)b^{n-m}\big)-\phi\big(R_mb^{n-m}\big)\Big)\nonumber\\
&=\sum_{m=1}^\infty \alpha ^{-m}\sum_{n=0}^{m-1} \beta^n\Big(\phi\big((R_m+1)b^{n-m}\big)-\phi\big(R_mb^{n-m}\big)\Big)\nonumber
\\
&=\sum_{m=1}^\infty \alpha ^{-m}\sum_{\ell=1}^{m} \beta^{m-\ell}\Big(\phi\big((R_\ell+1)b^{-\ell}\big)-\phi\big(R_\ell b^{-\ell}\big)\Big)\nonumber
\\
&=\frac1{1-\beta/\alpha}\sum_{\ell=1}^\infty\alpha^{-\ell}\Big(\phi\big((R_\ell+1)b^{-\ell}\big)-\phi\big(R_\ell b^{-\ell}\big)\Big).\label{kernel kernel alt Z eq}\end{align}
By assumption, the latter series is not $\bP_R$-a.s.~zero. This concludes the proof. 
\end{proof}

\begin{remark}In the context of \Cref{kernel kernel prop}, let $f(t):=\sum_{n=0}^\infty\alpha^n\psi(\{b^nt\})$. Then \Cref{p variation prop}  states that $\<f\>^{(p)}_t=t\cdot\mathbb E[|Z|^{p}]$ for $p=-\log_\alpha b$ and $Z$ as in \eqref{kernel kernel Z eq}. By \eqref{kernel kernel alt Z eq}, $Z$ can be represented as follows in term of $\phi$,
\begin{equation}\label{kernel kernel rem eq}
Z=\frac1{1-\beta/\alpha}\sum_{m=1}^\infty\alpha^{-m}\Big(\phi\big((R_m+1)b^{-m}\big)-\phi\big(R_m b^{-m}\big)\Big).
\end{equation}
Now consider the specific case in which $\phi(t):=t\wedge(1-t)$ is the tent map and $0<\beta<1/b<\alpha<1$. Then $\phi$ satisfies \eqref{cond} and hence the conditions of \Cref{kernel kernel prop} hold. Moreover, $\psi$ in \eqref{kernel kernel psi eq}
 is a Takagi--van der Waerden function. If in addition $b$ is even and $Z$ is as in \eqref{kernel kernel rem eq}, then the law of $b(1-\beta)Z$ is the infinite Bernoulli convolution with parameter $1/(\alpha b)$. This follows from Proposition 3.2 (a) in~\cite{SchiedZZhang}.
\end{remark}

 \parskip-0.5em\renewcommand{\baselinestretch}{0.9}\normalsize
\bibliography{CTbook}{}
\bibliographystyle{plain}

\end{document}